\documentclass[10pt]{amsart}

\usepackage{graphicx}
\usepackage{amsfonts, amssymb,amsmath, amsthm}
\usepackage{tikz}
\usepackage[pdftex, plainpages=false]{hyperref}

\usepackage{mathtools}
\DeclarePairedDelimiter{\ceil}{\lceil}{\rceil}
 
\theoremstyle{plain}
\newtheorem{thmi}{Theorem}
\newtheorem{defi}[thmi]{Definition}
\newtheorem{cori}[thmi]{Corollary}

\newtheorem{qui}[thmi]{Question}

\newtheorem*{conji}{Conjecture}

\newtheorem{theorem}{Theorem}[section]
\newtheorem{proposition}[theorem]{Proposition}
 \newtheorem{definition}[theorem]{Definition}

\newtheorem{example}[theorem]{Example}
\newtheorem{notation}[theorem]{Notation}
\newtheorem{lemma}[theorem]{Lemma}
\newtheorem{question}[theorem]{Question}
\newtheorem{corollary}[theorem]{Corollary}
\newtheorem{remark}[theorem]{Remark}

\newcommand{\growthp}[1]{\beta'(#1)}
\newcommand{\growthpp}[1]{\beta''(#1)}
\newcommand{\fffun}[1]{f_{#1}}

\newcommand{\Z}{\mathbb{Z}}   
\newcommand{\N}{\mathbb{N}}

\DeclareMathOperator{\dist}{\mathsf{dist}}

\begin{document}

\title[]{The Coarse Geometry of Hartnell's Firefighter Problem on Infinite Graphs}
\author[D.~Dyer]{Danny Dyer}
\author[E.~Mart\'inez-Pedroza]{Eduardo Mart\'inez-Pedroza}
\author[B.~Thorne]{Brandon Thorne}
  \address{Memorial University\\ St. John's, Newfoundland, Canada A1C 5S7}
  \email{dyer@mun.ca, emartinezped@mun.ca}

\begin{abstract}
In this article, we study Hartnell's Firefighter Problem through the group theoretic notions of growth and quasi-isometry.  A graph has the $n$-containment property if for every finite initial fire, there is a  strategy to contain the fire by protecting $n$ vertices at each turn. A graph has the constant containment property if there is an integer $n$ such that it has the $n$-containment property. Our first result is that any locally finite connected graph with quadratic growth has the constant containment property; the converse does not hold.  A second result is that in the class of graphs with bounded degree, having the constant containment property is closed under quasi-isometry.
We prove analogous results for the $\{f_n\}$-containment property, where $f_n$ is an integer sequence corresponding to the number of vertices protected at time $n$. In particular, we positively answer a conjecture by Develin and Hartke  by proving that the $d$-dimensional square grid $\mathbb{L}^d$ does not satisfy the $cn^{d-3}$-containment property for any constant $c$.  
\end{abstract}

\maketitle

\section*{Introduction}

The firefighter problem on graphs was introduced by Bert Hartnell in 1995 and it has been studied in the last two decades~\cite{FHS00, FM09}. Briefly, the game can be described as follows, and we refer the reader to Section~\ref{sec:preliminaries} for precise definitions. Let $G$ be a graph and let $f$ be a positive integer; an initial fire starts at a finite set of vertices; at each time interval $n\geq 1$, $f$ vertices which are not on fire become protected, and then the fire spreads to all unprotected neighbors of vertices on fire; once a vertex is protected or is on fire, it remains so for all time intervals. The graph $G$ has the \emph{$f$-containment property} if every initial fire admits a strategy to protect $f$ vertices at each  time interval so that the set of vertices on fire is eventually constant. We say that the graph $G$ has the \emph{constant containment property} if it has {the $f$-containment property} for some positive integer $f$.

The constant containment property is well-understood in certain grids of the Euclidean plane. For example,   the infinite ($2$-dimensional) square grid has the $2$-containment property~\cite{Fo03}, the $2$-dimensional infinite triangular grid has the $3$-containment property~\cite{Fo03, Me07},  the $2$-dimensional hexagonal grid has the $2$-containment property and  the ``strong'' grid has the $4$-containment property~\cite{Me04}, see Figure~\ref{fig:01}. Recent work includes other metrics on the game such as the surviving rate, the portion of the grid saved by fewer than the optimal number of firefighters~\cite{GKP14}.

\begin{figure}[t]
\center
\includegraphics[width=\linewidth]{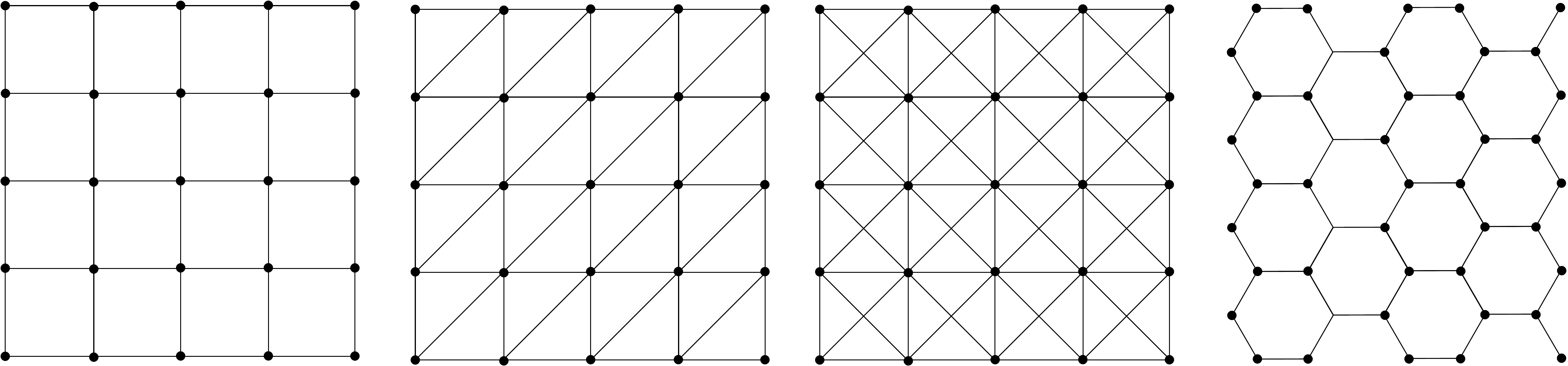}
\caption{The $2$-dimensional square grid, triangular grid, strong grid, and hexagonal grid.}
\label{fig:01}
\end{figure}

The $2$-dimensional grids mentioned above have quadratic growth in the following sense. Let $G$ be a connected graph and let $g_0$ be a chosen vertex; the \emph{growth function of $G$ based at $g_0$} is the function $\beta \colon \N \to \N$ such that $\beta(n)$ is the number of vertices of $G$ which are at distance at most $n$ from $g_0$.  We say that $G$ has \emph{polynomial growth of degree $d$} if there is $C>0$ such that $\beta (n) \leq Cn^d$. In particular, we say the growth function of $G$ is \emph{quadratic} if it has polynomial growth of degree $2$.   One can  verify that having polynomial growth of degree $d$ is independent of chosen vertex for connected graphs.  Growth functions of graphs have been studied in relation to discrete groups, for a brief overview and references we refer the reader to~\cite{BrHa99}. 

\begin{thmi}[Corollary \ref{cor:quadratic}]\label{theorem01}
If $G$ is a connected graph with quadratic  growth, then $G$ satisfies the constant containment  property.
\end{thmi}

Theorem~\ref{theorem01}  is illustrated by the fact that  the four $2$-dimensional grids in Figure~\ref{fig:01} have the constant containment property~\cite{Fo03, Me04, Me07}, see also~\cite{FH13, NR08, MW02}. 

It is known that the underlying graph of any uniform tiling  of the Euclidean plane has quadratic growth, for an account see Section~\ref{sec:tilings}.

\begin{cori}\label{cor:2d}
The underlying graph of any uniform tiling of the Euclidean plane has the constant containment property.
\end{cori}

The following variation of the constant containment property is implicit in work of Develin and Hartke~\cite{DeHa07}. Let $G$ be a graph and let $\{f_n\}$ be a sequence of integers;  an initial fire starts at a finite set of vertices; at each time interval $n\geq 1$, at most $f_n$ vertices which are not on fire become protected, and then the fire spreads to all unprotected neighbors of vertices on fire; once a vertex is protected or is on fire, it remains so for all time intervals. The graph $G$ has the \emph{$\{f_n\}$-containment property} if every initial fire admits a strategy consisting of  protecting at most $f_n$ vertices at the $n^{th}$ time interval so that the set of vertices on fire is eventually constant. We say that the graph $G$ has the \emph{$O(n^d)$-containment property} if there is a  constant $c\geq 1$ such that $G$ has {the $\{cn^d\}$-containment property}. See Section~\ref{sec:preliminaries} for a more rigorous definition. The following generalization of Theorem~\ref{theorem01} holds.

 \begin{thmi} [Theorem \ref{thm:polynomialgrowth}]\label{thmi:polgrowth}
Let $G$ be a connected graph with polynomial growth of degree at most $d$. Then $G$ satisfies the $O(n^{d-2})$--containment  property.
\end{thmi}

The converse of Theorem~\ref{theorem01} does not hold.  There is a connected graph such that every vertex has degree at most $4$, its growth function is not bounded by a polynomial function, and it satisfies the $1$-containment property, see Example~\ref{ex:subexponential}.  However, we expect the converse of Theorem~\ref{thmi:polgrowth} to hold for certain classes of graphs, see Question~\ref{que:ggt}.  
 
In~\cite{DeHa07}, Develin and Hartke   studied containment properties of the $d$-dimensional square grids $\mathbb L^d$; for a precise definition see Section~\ref{def:Ld}. In particular, they show that  $\mathbb L^d$  does not have the constant containment property if $d>2$, see~\cite[Theorem 8]{DeHa07}. One can verify that the growth function of $\mathbb L^d$ is bounded by a polynomial of degree $d$, see Remark~\ref{rem:snLd}. Hence the following corollary is immediate.

\begin{cori}\label{cori:Ld+}
The $d$-dimensional square grid has the $O(n^{d-2})$-containment property.
\end{cori}

The following  approach to a converse of Theorem~\ref{thmi:polgrowth}   holds under a regularity condition called \emph{homogeneous growth}.  Homogeneous growth is discussed in Section~\ref{sec:homogeneous}.

\begin{thmi}[Theorem~\ref{thm:hgrowth}] \label{thmi:homogeneous}
Let $G$ be a graph which has homogeneous growth with respect to a vertex $g_0$.  Let $s_n$ be the number of vertices at distance exactly $n$ from $g_0$, and let $\{f_i\}$ be a non-decreasing sequence. If $G$ has the $\{f_n\}$-containment property, then the series $ \sum_{n=1}^\infty \frac{f_n}{s_n}$ diverges.
\end{thmi}

It can be shown that any orthant of the $d$-dimensional square grid has homogeneous growth, see Proposition~\ref{prop:Ldsatisfies}. 

\begin{cori}[Corollary~\ref{cor:conj}]\label{cor:conji}\label{cori:Ld-}
Let $d$ and $q$ be positive integers.
If $\lim_{n\to \infty} \frac{n^q}{n^{d-2}}=0$, then $\mathbb{L}^d$ does not satisfy the $O(n^q)$-containment property. In particular, $\mathbb{L}^d$ does not satisfy the $O(n^{d-3})$-containment property.
\end{cori}
 
 Corollary~\ref{cori:Ld-}  generalizes Develin and Hartke's result that $\mathbb L^d$ does not have the constant containment property for $d>2$, it shows that  Corollary~\ref{cori:Ld+} is sharp, and positively answers the following conjecture raised by Develin and Hartke in the polynomial case.

\begin{conji}\cite[Conjecture 9]{DeHa07}
Suppose that $f (t)$ is a function on $\N$ with the property that $f (t)/t^{d-2}$ goes to $0$ as $t$ gets large. Then there exists some outbreak on $\mathbb{L}^d$ which cannot be contained by deploying $f (t)$ firefighters at time $t$.  A weaker conjecture would require $f (t)$ to be a polynomial.
\end{conji}

The main result of the paper is that containment properties on graphs are preserved by quasi-isometry, see Theorem~\ref{thm:qii}.  The notion of quasi-isometry is an equivalence relation between metric spaces which  plays a significant role in the study of discrete groups, for an overview see~\cite{BrHa99} and  references therein.  We consider graphs as metric spaces as follows. The notion of path defines a metric on the set of vertices of a graph $G$ by declaring $\dist_G(x, y)$ to be the length of the shortest path from $x$ to $y$. The metric $\dist$ on the set of vertices of $G$ is called \emph{the combinatorial metric on $G$.}

\begin{defi}\cite{BrHa99}\label{def:qi}
Let $(X_1, \dist_1)$ and $(X_2, \dist_2)$ be metric spaces. A (not necessarily continuous) map $f\colon X_1 \to X_2$ is called a \emph{$(\lambda, \epsilon, c)$-quasi-isometry} if $\lambda\geq 1$ and $\epsilon\geq 0$ are real numbers such that for all $x,y\in X_1$
\[ \frac1\lambda \dist_1 (x,y) -\epsilon \leq \dist_2 (f(x), f(y)) \leq \lambda \dist_1 (x,y)+\epsilon,\]
and  $c\geq 0$ is a real number such that  every point of $X_2$ lies in the $c$-neighborhood of the image of $f$. When such  a map exists, $X_1$ and $X_2$ are said to be \emph{$(\lambda,\epsilon, c)$-quasi-isometric}.
\end{defi}

\begin{thmi} [Theorem~\ref{thm:qi}] \label{thm:qii}
Let $G$ and $H$ be quasi-isometric graphs with bounded degree.
If $G$ satisfies the $O(n^d)$-containment property then  $H$ satisfies the $O(n^d)$-containment property.
\end{thmi}

Given a graph $G$ and an integer $k$, let  $G(k)$ be the graph having the same set of vertices as $G$ and such that two vertices are connected by an edge if they are at distance at most $k$ in $G$.  Observe that $G$ is $(k,0,0)$-quasi-isometric to $G(k)$ for any positive integer $k$. As a consequence, we have the following corollary.

 \begin{cori}\label{cor:gk}
 Let $G$ be a graph with bounded degree. Then $G$ has the $O(n^d)$-containment property if and only if $G(k)$ has the $O(n^d)$-containment property for every $k\geq 1$.
\end{cori}

Observe that Corollary~\ref{cor:2d} is also a consequence of Theorem~\ref{thm:qii},
since the underlying graphs of a uniform tiling of the Euclidean planes are all quasi-isometric; see Section~\ref{sec:groups} for a brief explanation.  
 
Analogously,  the underlying graph of a uniform tiling of the hyperbolic plane with the combinatorial metric is quasi-isometric to the $2$-dimensional hyperbolic space, see Section~\ref{sec:groups}.  One can prove directly that the underlying graph of the order-$7$ triangular tiling of the hyperbolic plane does not have a polynomial containment property, see Proposition~\ref{prop:expANDhom} and its corollary. Hence Theorem~\ref{thm:qii} implies the following statement.

\begin{cori}\label{cor:hyperbolic}
The underlying graph of any uniform tiling of the hyperbolic plane  does not have a polynomial containment property.
\end{cori}

By the infinite $\delta$-regular tree, we mean an infinite tree such that every vertex has degree exactly $\delta$. It is well known that any pair of infinite regular trees of degree at least $3$ are quasi-isometric~\cite[Page 141]{BrHa99}.
Theorems~\ref{thm:qii} and~\ref{thmi:homogeneous} yield the following result which provides a sufficient condition implying that graph does not have a polynomial containment property.

\begin{cori}[Corollary~\ref{cor:tree}] \label{cori:tree}
If a graph $H$  contains a subgraph  quasi-isometric to the infinite $\delta$-regular tree with $\delta\geq 3$, then $H$ does not satisfy a polynomial containment property.
\end{cori}

We remark that Corollary~\ref{cor:hyperbolic} can be verified via Corollary~\ref{cori:tree} by observing that   the underlying graph of the order-$7$ triangular tiling of the hyperbolic plane contains an infinite $3$-regular tree as a subgraph. In fact, one can prove that the underlying graph of any uniform tiling of the hyperbolic plane contains a subgraph quasi-isometric to an infinite $3$-regular tree.

\subsection*{Connection with geometric group theory.}
Properties of bounded degree graphs which are preserved under quasi-isometry are known as \emph{geometric properties}.  These types of properties define invariants of finitely generated groups, as discussed in Section~\ref{sec:groups}.  This follows from the observation that for any finitely generated group, Cayley graphs  associated to different finite generating sets are quasi-isometric.
A well-known example of a geometric property is having growth of degree $d$, and our main result Theorem~\ref{thm:qii} states that the $O(n^d)$-containment property is also geometric.

A finitely generated group $G$ has \emph{growth of degree $d$}  if there is a finite generating set of $G$ for which the  corresponding Cayley graph has growth of degree $d$. Analogously, \emph{$G$ has the $O(n^d)$-containment property} if there is a finite generating set of $G$ for which the  corresponding Cayley graph has the $O(n^d)$-containment property.

By Theorem~\ref{thmi:polgrowth}, if a finitely generated group has polynomial growth then it has the polynomial containment property.  Since the Cayley graph of a non-cyclic free group is quasi-isometric to a tree,  Corollary~\ref{cori:tree} implies that groups containing non-cyclic free groups do not have a polynomial containment property.  
We expect the following question to have a positive answer. However, in view of Example~\ref{ex:subexponential}, answering the question requires more than coarse geometry techniques.

\begin{qui}\label{que:ggt}
In the class of finitely generated groups, is having quadratic growth equivalent to having constant containment property? More generally,  is having polynomial growth of degree $d$ equivalent to having the $O(n^{d-2})$--containment property?
\end{qui}

A positive answer to this question would characterize virtually nilpotent groups in terms of Hartnell's firefighter games via Gromov's polynomial growth theorem~\cite{Gr81}. Some remarks on this question are discussed in~\cite{MP17}.

\subsection*{Outline.} The rest of the article is divided into five sections. These sections are independent of the introduction and, in particular, statements and definitions are re-introduced.  Section~\ref{sec:preliminaries} introduces language, notation, and some preliminary results. In particular, Proposition~\ref{prop:expANDhom} which provides sufficient conditions for a graph with exponential growth to not satisfy a polynomial containment property.  The proofs of Theorem~\ref{thmi:polgrowth} on graphs with polynomial growth and its corollaries constitute Section~\ref{sec:growth}.  Section~\ref{sec:homogeneous} introduces the notion of homogeneous growth and discusses the proof of Theorem~\ref{thmi:homogeneous} and its corollaries. The main result of the article, Theorem~\ref{thm:qii}, is proved in Section~\ref{sec:qi}.  Section~\ref{sec:groups} contains a brief discussion on considering containment properties as  quasi-isometry invariants of finitely generated groups, and as an application we deduce the results on tilings of the Euclidean and Hyperbolic planes in this context.  

\subsection*{Acknowledgments.} We thank Bojan Mohar for comments on a preliminary version of the article, in particular, for pointing out Corollary~\ref{cor:mohar} and suggesting Question~\ref{que:mohar}.  Both Dyer and Mart\'inez-Pedroza acknowledge funding by the Natural Sciences and Engineering Research Council of Canada, NSERC.

\section{Containment properties on graphs, preliminaries} \label{sec:preliminaries}

Let $G$ be an undirected graph. A \emph{path of length $n$} is a sequence of vertices $v_0,v_1, \ldots , v_n$ such that $v_i, v_{i+1}$ are connected by an edge for each $i<n$.  The notion of path defines a metric on the set of vertices of $G$ by declaring $\dist_G(x, y)$ to be the length of the shortest path from $x$ to $y$. The metric $\dist$ on the set of vertices of $G$ is called \emph{the combinatorial metric on $G$.} In this note, when we consider  a graph as a metric space, we mean its set of vertices with the combinatorial distance. For a subset $X$ of $G$, the \emph{ball of radius $r$ about $X$,} $B_G(r, X)$ is defined as the collection of vertices at distance less than or equal to $r$ from at least one vertex in $X$. Analogously, for a vertex $g_0$, the \emph{sphere of radius $r$ about $g_0$}, $S_G(r, g_0)$ is the set of vertices at distance exactly $r$ from $g_0$; when $G$ and $g_0$ are understood, we use $S_r$ to denote this set, and $s_r$ to denote its cardinality.
A graph has \emph{locally finite} if every vertex has finite degree or equivalently every ball with unit radius centered at a vertex is a finite set.
A graph has \emph{bounded degree} if there is an upper bound on the cardinality of closed balls of unit  radius centered at vertices, or equivalently there is a finite maximum degree.

Given a sequence of integers $\{f_n\}$ and a graph $G$ we consider the following game. Suppose that a fire breaks out at a finite set of vertices $X_0$. At each subsequent time unit $n$ (called a turn), the player chooses  a set $W_n$ of at most $f_n$ distinct  vertices not on fire to become protected; then the fire spreads to all  adjacent vertices  which are on fire and are not yet protected. Once a vertex is on fire or is protected, it stays in such state for all subsequent turns.  If eventually the set of vertices on fire remains constant we say that the fire has been contained.  If every fire can be contained we say the graph $G$ has the $(\{f_n\}, 1)$-containment property.  A slight variation of the game is defined by allowing the fire to spread to all vertices which are connected by paths of length at most $r$ from a vertex on fire.  If every fire can be contained in this more general version, we say that the graph has the $(\{f_n\},r)$-containment property.  The following definition make these properties of graphs precise.

\begin{definition}\label{def:firegame2} Let $G$ be a graph, let $r$ be a positive integer, and let $\{f_n\}$ be a sequence of non-negative integers.  Given a finite subset $X_0$ of vertices of $G$, a sequence $\{W_k\colon k\geq 1\}$ of subsets of vertices of $G$ is a \emph{$(\{f_n\},r)$-containment strategy for $X_0$} if
\begin{enumerate}
\item for every $n\geq 1$, the set $W_n$ has cardinality at most $f_n$,
\item the sets $X_{n}$ and $W_{n+1}$ are disjoint for $n\geq 0$, where  $X_n$ for $n>0$ is defined as the set of vertices which are connected to a vertex in $X_{n-1}$ by a path of length at most $r$ containing no vertices in $W_1\cup \cdots \cup W_n$, and
\item \label{item:time} there is $N>0$ such that $X_n=X_N$ for every $n\geq N$.
\end{enumerate}
In this case, the set $X_0$ is called the \emph{initial fire}, and the sets  $W_n$ and $X_n$ are called the \emph{set of vertices protected at time $n$} and \emph{the set of vertices on fire at time $n$} respectively.  The integer $N$ is called a \emph{sufficient time} to contain the initial fire $X_0$, in general, we will choose a minimal $N$.

If for any finite subset of vertices $X_0$ of $G$ there exists a  $(\{f_n\},r)$-containment strategy then we shall say that \emph{$G$ has the $(\{f_n\},r)$-containment property}.
\end{definition}

\begin{notation}
By the $\{f_n\}$-containment property we mean the $(\{f_n\},1)$-containment property. We say that a graph $G$ satisfies $O(n^d)$-containment property if $G$ has the $\{f_n\}$-containment property for some sequence $\{f_n\}$ which is $O(n^d)$. In the case $d=0$, we say that $G$ has the \emph{constant containment property}. Analogously, we say that $G$ satisfies a \emph{subexponential containment property} if it satisfies the $\{f_n\}$-containment property for a sequence $\{f_n\}$ such that $\lim_{n\to \infty} \frac{\log f_n}{n} =0$ and  $\lim_{n\to \infty} \frac{f_n}{n^d} = \infty$ for every $d\geq 0$.
 \end{notation}

\begin{remark}
Suppose $\{W_k\colon k\geq 1\}$ is a $(\{f_n\},r)$-containment strategy for $X_0$ such that $X_n=X_N$ for every $n\geq N$. Then $X_n \subseteq B_G(X_0, rN)$ for every $n\geq N$.  Moreover, one can assume that $W_k=\emptyset$ for every $k\geq N$.  In particular containment strategies can be assumed to be finite sequences.
\end{remark}

From here on, we only consider locally finite graphs, i.e. graphs such that every vertex has finite degree. In this class of graphs, Definition~\ref{def:firegame2} can be re-stated as indicated in Proposition~\ref{prop:locfin}.

\begin{proposition}\label{prop:locfin}
For the class of locally finite graphs, in Definition~\ref{def:firegame2}, replacing the statement~\eqref{item:time} by
\begin{itemize}
\item[$(3')$] there is $N>0$ such that $X_n \subseteq B_G(X_0, rN)$ for every $n\geq 0$
\end{itemize}
yields an equivalent definition of $(\{f_n\},r)$-containment strategy for $X_0$.
\end{proposition}
\begin{proof}
Let $G$ be a locally finite graph, and let $r$ be a positive integer.  Let $X_0$ be a set of vertices, let $\{W_k\colon k\geq 1\}$ be a sequence of vertex sets, and define inductively $X_{n+1}$ to be the set of vertices $v$ such that there is path $\gamma$ of length at most $r$ from a vertex of $X_n$ to $v$ such that $\gamma$ does not contain vertices in $W_1\cup \cdots \cup W_{n+1}$.

Suppose there is $N>0$ such that $X_n=X_N$ for every $n\geq N$. Since $X_N\subseteq B_G(X_0, rN)$ and $X_n \subseteq X_{n+1}$ for every $n$, it is immediate that $X_n \subseteq B_G(X_0, rN)$ for every $n\geq 0$.

Conversely, suppose there is $N>0$ such that $X_n \subseteq B_G(X_0, rN)$ for every $n\geq 0$.
The assumption that $G$ is locally finite implies that  $B_G (X_0, rN)$ is a finite set of vertices.
Since $X_n \subseteq X_{n+1}$ for every $n$, it follows that there is $M>0$ such that $X_n=X_M$ for every $n\geq M$.
\end{proof}

\begin{example}
The locally finite hypothesis in Proposition~\ref{prop:locfin} is necessary. Let $G$ be  the undirected graph with vertex set $V= \Z \cup \{\infty\}$ and edge set $E=\{ (n, n+1) \colon n \in \Z\} \cup \{(n, \infty) \colon n\in \Z\}$.  Consider the initial fire $X_0=\{ 0\}$, let $W_1=\{ \infty \}$ and let $W_n = \emptyset$ for $n>1$. Observe that $X_n=\{0,1,-1, \ldots , n, -n\}$. Hence, $X_n \subsetneq X_{n+1}$ and $X_n \subset B_G(X_0, 2)$  for every $n$. In particular,  $\{W_n\colon n\geq 1\}$ is a not a containment strategy for $X_0$ in the sense of Definition~\ref{def:firegame2}.
\end{example}

\begin{remark}\label{rem:reduction} Let $G$ be  locally finite graph.
If $Y_0\subseteq X_0$ are finite subsets of vertices of $G$ and $\{W_n\}_{n\geq 1}$ is a {$(\{f_n\},r)$-containment strategy for $X_0$}, then $\{W_n\}_{n\geq 1}$ is also a {$(\{f_n\},r)$-containment strategy for $Y_0$}.  In particular,  if for every vertex $g\in G$ and for every integer $n\geq 0$ there is {$(\{f_n\},r)$-containment strategy for the ball $B_G(g, n)$}, then $G$  has the $(\{f_n\},r)$-containment property.
\end{remark}

\begin{proposition}
\label{scalingTurns}
Let $\{f_n\}$ be a non-decreasing sequence of non-negative integers, let $r$ be a positive integer, and let $g_{n+1}=\sum_{i=1}^r f_{rn+i}$. The graph $G$ has the $(\{f_n\},1)$-containment property if and only if $G$ has the $(\{g_n\},r)$-containment property.
\end{proposition}
\begin{proof}
If $G$ has the $(\{g_n\},r)$-containment property, then it has the $(\{f_n\},1)$-containment property. Indeed, let $X_0$ be a finite set of vertices. Suppose that $\{W_n \colon n\geq 1\}$ is a $(\{g_n\},r)$-containment strategy for $Y_0=B_G(X_0, r)$, and let $Y_n$ be the set of vertices on fire at time $n$. Since the cardinality of $W_{n+1}$ is at most $g_{n+1}=\sum_{i=1}^r f_{rn+i}$,  we can choose a partition $W_{n+1}=W_{n,1}\cup \cdots \cup W_{n, r}$ where $W_{n,i}$ has at most $f_{rn+i}$ elements. For $k\geq 0$ and $1\leq i\leq r$, define $U_{rk+i}$ to be the set $W_{k, i}$. We claim that   $\{U_n\colon n\geq 1\}$ is a $(\{f_n\}, 1)$-containment strategy for $X_0$.
First observe that for each $n\geq 0$, the sets $U_{rn+1}, \cdots, U_{rn+r}$ are disjoint from $B_G(X_0, r(n+1))$, and since $X_{nr+i} \subseteq B_G(X_0, nr+i)$ it follows that  $U_{rn+i}$ and $X_{rn+i}$ are disjoint. Observe that $X_{(k+1)r} \subseteq Y_k$ and since $Y_n$ is eventually constant, the sequence $X_n$ is eventually constant.

Conversely, if $G$ has the $(\{f_n\},1)$-containment property, then it has the $(\{g_n\},r)$-containment property. Indeed, if $\{W_n \colon n\geq 1\}$ is a $(\{f_n\},1)$-containment strategy for $X_0$  then $\{U_n\colon n\geq 1\}$ where
$U_n = W_{(n-1)r+1} \cup \cdots \cup W_{nr}$
is a $(\{g_n\}, r)$-containment strategy for $X_0$.
\end{proof}

\begin{corollary}
If $\{f\}$ is a constant sequence then the graph $G$ has the $(\{f\},1)$-containment property if and only if $G$ has the $(\{rf\}, r)$-containment property for every $r\geq 1$.
\end{corollary}

\begin{proposition} \label{prop:subgraph}
Let $H$ be a locally finite graph.  If $H$ has the $\{f_n\}$-containment property and $G$ is a subgraph of $H$ then $G$ has the $\{f_n\}$-containment property.
\end{proposition}
\begin{proof}
 Let $X_0$ be a finite subset of $G$. Let $\{U_k\colon k\geq 1\}$ be a $(\{f_n\},1)$-containment strategy for the initial fire $X_0$ in the graph $H$. Let $Y_k$ be the set of vertices on fire in $H$ at time $k$. Suppose that $Y_k \subseteq B_H(X_0, N)$ for $k\geq N$.  A $(\{f_n\},1)$-containment strategy $\{W_k\colon k\geq 1\}$ for $X_0$ in $G$ is defined as follows. Let $W_k$ be $U_k\cap G$.  Let $X_k$ be the set of vertices on fire in $G$ at time $k$.  Observe that $X_k \subseteq Y_k \cap G$, but   equality does  not hold in general.   It follows that $X_k \subseteq B_H(X_0, N)\cap G$ for $k\geq N$. Since all vertices of $H$ have finite degree, it follows that $B_H(X_0, N)\cap H$ is a finite set of vertices and hence there is $M\geq N$ such that $X_k \subseteq B_G(X_0, M)$ for $k\geq M$.
\end{proof}

\begin{lemma}\label{lem:easyineq}
Let $\{s_n\}$ be a non-decreasing positive sequence, and let $\{f_n\}$ and $\{p_n\}$ be positive sequences. Suppose $\sum_{i=1}^k p_i \leq \sum_{i=1}^k f_i$ for every $k$. Then $\sum_{i=1}^k \frac{p_i}{s_i} \leq \sum _{i=1}^k \frac{f_i}{s_i}$ for every $k$.
\end{lemma}
\begin{proof}
Since $s_n$ is non-decreasing and positive, and $\sum_{i=1}^j p_i \leq \sum_{i=1}^j f_i$ for $1\leq j\leq k$, we have that
\begin{equation}\nonumber
\begin{split}
\frac{p_1}{s_1}+\frac{p_2}{s_2}+ & \frac{p_3}{s_3} + \ldots  +  \frac{p_k}{s_k}  \leq
   \frac{f_1}{s_1}+\frac{p_1+p_2-f_1}{s_2}+\frac{p_3}{s_3} + \ldots + \frac{p_k}{s_k}    \\
\leq & \frac{f_1}{s_1}+\frac{f_2 }{s_2}+\frac{p_1 + p_2 +p_3   -f_1-f_2}{s_3}  + \frac{p_4}{s_4} + \ldots +\frac{p_{k-1}}{s_{k-1}}+ \frac{p_k}{s_k}   \\
  \cdots \\
   \leq  &\frac{f_1}{s_1}+\frac{f_2 }{s_2}+\frac{f_3}{s_3} + \ldots +\frac{f_{k-1}}{s_{k-1}}+ \frac{p_1 + \cdots +p_k  - (f_1+ \cdots +f_{k-1})}{s_k}     \\
   \leq  & \frac{f_1}{s_1}+\frac{f_2 }{s_2}+\frac{f_3}{s_3} + \ldots +\frac{f_{k-1}}{s_{k-1}}+ \frac{f_k}{s_k}.
\end{split}
\end{equation}
\end{proof}

\begin{proposition}\label{prop:expANDhom}
Let $G$ be a locally finite graph. Suppose that there are a vertex $g_0$ and a real number $\lambda>1$ such that for any integer $n\geq 0$ and any subset of vertices $A$ of the sphere $S_G(g_0, n)$,
\[ |A^*| \geq \lambda |A|,\]
where $A^*$ is the set of vertices of the sphere $S_G(g_0, n+1)$ which are adjacent to a vertex in $A$.
Then for any $d\geq 0$, the graph $G$ does not have the $O(n^d)$-containment property.
\end{proposition}

\begin{proof}
Suppose there exist constants $c>0$ and $d>0$ so that $G$ has the $\{cn^d\}$-containment property.

Let $S_n$ denote the sphere of radius $n$ about $g_0$,  and let $s_n$ denote its cardinality. Observe that $s_n \geq \lambda^n$.  Let $r$ be a positive integer such that
\begin{equation}\label{eq:seriespe} s_r >  c   \sum_{k=1}^\infty   \frac{k^d}{\lambda^k}  .\end{equation}

Let $\{W_k\colon k\geq 1\}$ be a $\{cn^d\}$-containment strategy for the initial fire $X_0=B(g_0, r)$,   let $X_k$ denote the set of vertices on fire at time $k$, and suppose that the fire is contained after time $\ell$, i.e., $X_k=X_{\ell-1}$ for every $k \geq \ell$. In particular, for every $k\geq 0$, we have that
\begin{equation} \label{eq:endoffire} X_k \subseteq B(g_0, r+\ell).\end{equation}

For $k\geq 0$, let
\begin{equation} T_k = S_{r+k}\cap X_k,\end{equation}
 equivalently, $T_k$ is the set of vertices on fire in the sphere $S_{r+k}$ at time $k$. In particular,  $T_{k+1}$ is the set of vertices in $T_{k}^* \subset S_{r+k+1}$ which remain unprotected at time $k$. Since $X_k=X_{\ell-1}$ for every $k \geq \ell$, we have that
 \begin{equation} T_\ell = \emptyset. \end{equation}

Let $p_{k+1} = |T_k^* \setminus  T_{k+1}|$, and observe that $p_k$ is the number of vertices of $S_{r+k}$ that are protected up to time $k$.  The hypothesis on $G$ implies
\begin{equation} |T_k^*| \geq \lambda |T_k|  \end{equation}
for every $k\geq 0$. Therefore
\begin{equation} \label{eq:37ssubfinal}
\begin{split}
 |T_{\ell}| = |T_{\ell-1}^*| - p_{\ell}  & \geq  \lambda|T_{\ell-1}| - p_{\ell} \\
 & \geq  \lambda^2 |T_{\ell-2}| -\lambda p_{\ell-1} -p_{\ell} \\
& \quad \quad \vdots \\
& \geq \lambda^{\ell} |T_{0}| - \sum_{k=1}^{\ell} \lambda^{\ell-k} p_{k}.
\end{split}
  \end{equation}

 A maximum of $\sum_{k=1}^m c k^d$ vertices    can be protected within $m$ turns. Therefore \begin{equation}\sum_{k=1}^m p_k \leq \sum_{k=1}^m c k^d \end{equation} for every $m$.
By Lemma~\ref{lem:easyineq},
\begin{equation} \label{eq:lemma}  \sum_{k=1}^{\ell} \frac{ p_k}{\lambda^k} \leq   c\sum_{k=1}^{\ell}  \frac{ k^d}{\lambda^k}. \end{equation}

Since $|A_0|=s_r$, we have that
\begin{equation} \label{eq:37subfinal}
\begin{split}
\lambda^{\ell} |T_{0}| - \sum_{k=1}^{\ell} \lambda^{\ell-k} p_{k} & =   \lambda^{\ell} s_r \left (1 - \frac1{s_r}\sum_{k=1}^{\ell} \frac{p_k}{\lambda^k} \right)\\
 & \geq \lambda^{\ell} s_r \left( 1 -    \frac{c}{s_r}\sum_{k=1}^{\ell}   \frac{k^d}{\lambda^k}\right)  \\
& >0,
\end{split}
\end{equation}
where
the first inequality follows from~\eqref{eq:lemma},
 and the last inequality follows from~\eqref{eq:seriespe}.

Putting together inequalities~\eqref{eq:37ssubfinal} and~\eqref{eq:37subfinal} yields that
\begin{equation} T_\ell \neq \emptyset. \end{equation}
We have reached a contradiction and therefore the assumption that $G$ has the $\{cn^d\}$-containment property is false.
\end{proof}

\begin{remark}
The following observation was pointed out by a referee of the article. Under the hypotheses of Proposition~\ref{prop:expANDhom}, the stronger statement holds: If $\{f_n\}$ is a sequence such that the series $\sum_{k=1}^\infty   \frac{f_k}{\lambda^k}$ converges, then the graph $G$ does not have the $\{f_n\}$-containment property. Indeed, the same proof works by choosing $r$ such that
\[ s_r >     \sum_{k=1}^\infty   \frac{f_k}{\lambda^k}.
\]
instead of~\eqref{eq:seriespe}.
\end{remark}

The following is an interesting example of a graph that does not satisfy a polynomial containment property.

\begin{corollary}
\label{prop:37tilingCP}
The underlying graph $G$  of the order-$7$ triangular tiling of the hyperbolic plane does not have a polynomial containment property.
\end{corollary}
\begin{proof}
It is enough to verify that  the graph $G$ satisfies the hypothesis of Proposition~\ref{prop:expANDhom}.
Let $g_0$ be a vertex of $G$, let $B_n$ denote the subgraph expanded by the collection of vertices at distance at most $n$ from $g_0$, and let $S_n$ denote  the subgraph of $G$ expanded by the collection of vertices at distance  exactly $n$ from $g_0$.  For $n\geq 1$, an induction argument shows that $S_n$ is connected and every vertex has degree $2$, i.e. it is a cycle. Moreover, any vertex of $S_n$ has degree $3$ or $4$ as a vertex of $B_n$.  Let $s_n$, $a_n$, and $b_n$ denote the cardinality of $S_n$, the number of vertices of $S_n$ having degree $3$ in $B_n$, and the number of vertices of $S_n$ having degree $4$ in $B_n$ respectively. Then one observes that the following relations hold for $n\geq 1$:
\begin{equation}
s_n = a_n+b_n,  \quad  a_{n+1}= 2a_n+b_n,  \quad  b_{n+1}=s_n, \quad s_1=a_1=7, \quad b_1=0.
\end{equation}
An induction argument shows that $s_n=7f_{2n}$ and $a_n=7f_{2n-1}$ where $f_n$ is the Fibonacci sequence with $f_0=f_1=1$, which is a well known property of this tiling.

Let $A$ be a subset of vertices of $S_n$, and let $A^*$ be the set of vertices of $S_{n+1}$ which are adjacent in $G$  to a vertex in $A$. In $G$, each vertex $v$ of $A$ is adjacent to at least $3$ vertices of $A^*$. Indeed, if $v \in A$ has degree $3$ as a vertex of $B_n$, then $v$ is adjacent to $4$ vertices of $A^*$; if $v \in A$ has degree $4$ as a vertex of $B_n$, then $v$ is adjacent to $3$ vertices of $A^*$. It follows that for every $n$, and for every subset of vertices $A$ of $S_n$, we have that
\begin{equation} |A^*| \geq 2 |A| \end{equation} for every $A \subset S_n.$
\end{proof}

\section{The growth function and containment properties}\label{sec:growth}

\begin{definition} \label{def:growth}  Let $G$ be a locally finite graph and let $g_0$ be a vertex of $G$.  The \emph{growth function  of $G$ based at $g_0$} is the function $\beta\colon \N \to \N$ where $\beta_n$ is the number of vertices of $G$ at distance at most $n$ from $g_0$. The graph $G$ has \emph{polynomial growth of degree at most $d$} if there exists a constant $c$ such that $\beta_n \leq cn^d$ for all $n\geq 1$.
\end{definition}

\begin{remark}\label{rem:growth}
Let $g_0,g_1$ be vertices of $G$ at distance $k$. The corresponding growth functions satisfy
$ \beta_{g_0}(n) \leq \beta_{g_1}(n+k)$, since $B_G(g_0, n) \subseteq B_G(g_1, n+k)$. In particular, for a locally finite connected graph having polynomial growth of degree at most $d$ is independent of the base point.
\end{remark}

\begin{theorem} \label{thm:polynomialgrowth}
Let $G$ be a connected graph with polynomial growth of degree $d\geq 2$. Then $G$ satisfies the $O(n^{d-2})$--containment  property.
\end{theorem}

The strategy of the proof is to show that given an initial (finite) fire $X_0$ and a vertex $g_0\in G$, one can choose an integer $r>0$ such that $X_0 \subset B_G(g_0, r)$ and all the vertices of the  sphere $S_r=\{g\in G \colon \dist (g_0, g)=r\}$ can be protected before the fire reaches them. Then it follows that the fire cannot extend beyond distance $r$ from $g_0$ and hence it has been contained.

\begin{proof}[Proof of Theorem~\ref{thm:polynomialgrowth}]
Let $\beta\colon \N\to\N$ be the growth function of $G$ based at the vertex $g_0$. Let $c>0$ such that $\beta_n\leq cn^d$ for every $n\geq 1$. Let $s_{n+1}$ denote the difference $\beta_{n+1}-\beta_n$ for $n\geq 0$, let $s_0=1$, and observe that $\beta_n=\sum_{k=0}^ns_k$.

Recall Faulhaber's formula
\begin{equation}p_{n,d} = \sum_{k=1}^n k^{d-1}
= \frac1d n^d + \frac12n^{d-1} + \frac1d \sum_{j=2}^{d-1}  {d \choose j} B_j n^{d -j}  ,
\end{equation}
where $B_j$ denotes the $j$-th  Bernoulli number.

Observe that for every positive integer $m>0$ there is $N>0$ such that
 \begin{equation}\label{eq:polynomial2}
 d c n^{d-1} \leq (d-1)(dc+1) p_{n-m, d-1}
 \end{equation}
for every $n\geq N$. This follows from the observation that  both expressions are polynomials in $n$ of degree $d-1$ with leading coefficients $dc$ and $dc+1$ respectively.

We claim that there are infinitely many integers $n$ such that
\begin{equation}\label{eq:polynomial3} s_n <  d c n^{d-1}.\end{equation}
Suppose not, then there is $m>0$ such that $d c n^{d-1} \leq s_n$ for every $n> m$. Then
\begin{equation}\label{eq:degreed}
dcp_{n,d}- dcp_{m,d} + \beta_{m}   =   \sum_{k=m+1}^n dck^{d-1}  +  \beta_{m } \leq
\sum_{k=m+1}^n s_k +  \beta_{m} = \beta_n \leq cn^d
  \end{equation}
for every $n>m$. Since $c>0$ and the expression on the left of~\eqref{eq:degreed} is a polynomial in $n$ with leading terms $cn^d + \frac12dcn^{d-1}$, we have that~\eqref{eq:degreed} cannot hold for every $n>m$ and we have reached a contradiction. Therefore the claim holds.

From the statements on inequalities~\eqref{eq:polynomial2}
and~\eqref{eq:polynomial3}, we have that for every $m>0$, the inequality
\begin{equation}\label{eq:polynomial5}
s_n \leq (d-1)(dc+1) p_{n-m, d-1}
\end{equation}
holds for infinitely many values of $n$.

Let \begin{equation}
f_n =  (d-1)(dc+1) n^{d-2}.
\end{equation}
We claim that $G$ has the $(\{f_n\}, 1)$-containment property. Let $m>0$ and let $X_0$ be $B_G(g_0, m)$.  Since~\eqref{eq:polynomial5} holds for infinitely many $n$, choose a positive integer $r$  such that $s_r\leq (d-1)(dc+1) p_{r-m, d-1}$.
Since the set $S_r=\{g\in G\colon \dist(g_0, g)=r\}$ has cardinality $s_r$, it admits a partition  $W_1\cup \cdots \cup W_{r-m}$  such that  $W_k$ has at most $ (d-1)(dc+1) k^{d-2}$ elements. Hence $|W_k|\leq f_k$ for $1\leq k\leq r-m$. It follows that $\{W_k\colon k=1,\ldots ,r-m\}$ is a $(f_n, 1)$-containment strategy for $X_0$, since under such strategy $X_k=B_G(g_0, m+k)$ for $k<r-m$, and $X_k=B_G(g_0, r-m-1)$ for every $k\geq r-m$.
\end{proof}

Below we state a corollary of the proof of Theorem~\ref{thm:polynomialgrowth}.
\begin{corollary}\label{cor:mohar}
Let $G$ be a locally finite connected graph.
Let $g_0$ be a vertex of $G$ and let $s_n$ be the number of vertices at distance exactly $n$ from $g_0$.  If $\liminf_{n\to \infty} \frac{s_n}{n^{d-1}}$ is finite then $G$ satisfies the $O(n^{d-2})$--containment  property.
 \end{corollary}
\begin{proof}
To reach the conclusion, it is enough to verify that the statements on inequalities~\eqref{eq:polynomial2} and~\eqref{eq:polynomial3} in the proof of Theorem~\ref{thm:polynomialgrowth} hold. Observe that the statement that  for every positive integer $m>0$ there is $N>0$ such that inequality~\eqref{eq:polynomial2} holds for every $n\geq N$ is independently of the graph $G$. On the other hand, if $\liminf_{n\to \infty} \frac{s_n}{n^{d-1}}=c$, then there are infinitely many integers $n$ for which  inequality~\eqref{eq:polynomial3} holds.
\end{proof}

\begin{question}[Bojan Mohar]\label{que:mohar}
Does the converse of Corollary~\ref{cor:mohar} hold?
\end{question}

\begin{corollary} \label{cor:quadratic}
Let $G$ be a connected graph with quadratic  growth. Then $G$ satisfies the constant containment  property.
\end{corollary}

The converse of Corollary~\ref{cor:quadratic} does not hold even in the class of bounded degree graphs as the following examples illustrate.

\begin{example} \label{ex:subexponential}
There is a graph such that every vertex has degree at most four, it has subexponential growth, and it has the $(1,r)$-containment property for every $r\geq 1$.  Indeed, consider the sequence $\{s_n\}_{n\in\N}$
\begin{equation} 0, 1, 0, 1, 2, 1, 0 , 1, 2, 3, 2, 1, 0,  1, 2, 3, 4, 3, 2, 1, 0, \cdots .\end{equation}
Let $G$ be the graph with vertex set
\begin{equation}V=\left\{v_{n,x} \mid n\in\{0,1, \ldots \}   \mbox{ and } x\in \{1, \ldots , 2^{s_n} \} \right\},\end{equation}
and edge set
\begin{equation}E=\left\{ (v_{n,x}, v_{m, y} ) \mid  |m-n|=1 \mbox{ and } s_m-s_n=1 \mbox{ and } 2x-y \in \{0,1\} \right\}.\end{equation}
This graph is illustrated in Figure~\ref{fig:2sn}. Its growth function $\beta$ based at $v_{0,1}$ satisfies
\begin{equation}\nonumber \begin{split}
\beta (n(n+1)) & = 1 + \sum_{k=1}^n \left [ 2 \sum_{i=0}^{k} 2^i   - 2^{k} -1 \right ] = 3\cdot 2^{n+1} - 3n - 5 \\
\end{split} \end{equation}
since the number of vertices between $v_{k(k-1), 1}$ and $v_{k(k+1),1}$, including them, is exactly
$2 \sum_{i=0}^{k} 2^i   - 2^{k}$. In particular, $\beta (n^2)$ is roughly $2^{n}$.
\end{example}

\begin{figure}
\center

\begin{tikzpicture}[yscale=0.5]

\draw (0,0) node[left] {$v_{0,1}$} -- (0.5,0.5) node[above] {$v_{1,2}$} -- (1,0) ;
\draw (0,0) -- (0.5,-0.5) node[below] {$v_{1,1}$} -- (1,0);
\draw [fill] (0,0) circle [x radius=0.05,y radius=0.1];
\draw [fill] (0.5,0.5) circle [x radius=0.05,y radius=0.1];
\draw [fill] (0.5,-0.5) circle [x radius=0.05,y radius=0.1];
\draw [fill] (1,0) circle [x radius=0.05,y radius=0.1];
\draw (1,0) -- (1.5,0.75) node[above] {$v_{3,2}$} -- (2,1.25) node[above] {$v_{4,4}$} -- (2.5,0.75)  -- (3,0);
\draw (1.5,0.75) --  (2,0.25) -- (2.5,0.75);
\draw (1,0) -- (1.5,-0.75) -- (2,-1.25) node[below] {$v_{4,1}$} -- (2.5,-0.75) -- (3,0);
\draw (1.5,-0.75) node[below] {$v_{3,1}$} --  (2,-0.25) -- (2.5,-0.75) ;
\draw [fill] (1.5,0.75) circle [x radius=0.05,y radius=0.1];
\draw [fill] (1.5,-0.75) circle [x radius=0.05, y radius=0.1];
\draw [fill] (2,1.25) circle [x radius=0.05, y radius=0.1];
\draw [fill] (2,0.25) circle [x radius=0.05, y radius=0.1];
\draw [fill] (2,-0.25) circle [x radius=0.05, y radius=0.1];
\draw [fill] (2,-1.25) circle [x radius=0.05, y radius=0.1];
\draw [fill] (2.5,0.75) circle [x radius=0.05, y radius=0.1];
\draw [fill] (2.5,-0.75) circle [x radius=0.05, y radius=0.1];
\draw [fill] (3,0) circle [x radius=0.05, y radius=0.1];
\draw (3,0) -- (3.5,1.5) node[above] {$v_{7,2}$} -- (4,2.25) -- (4.5,2.75) node[above] {$v_{9,8}$} -- (5,2.25) -- (5.5,1.5) -- (6,0);
\draw (3.5,1.5) -- (4,0.75) -- (4.5,.25) -- (5,0.75) -- (5.5,1.5);
\draw (4,0.75) --  (4.5,1.25) -- (5,0.75);
\draw (4,2.25) --  (4.5,1.75) -- (5,2.25);
\draw (3,0) -- (3.5,-1.5) node[below] {$v_{7,1}$} -- (4,-2.25) -- (4.5,-2.75) node[below] {$v_{9,1}$} -- (5,-2.25) -- (5.5,-1.5) -- (6,0) node[right] {$v_{12,1}$};
\draw (3.5,-1.5) -- (4,-0.75) -- (4.5,-.25) -- (5,-0.75) -- (5.5,-1.5);
\draw (4,-0.75) --  (4.5,-1.25) -- (5,-0.75);
\draw (4,-2.25) --  (4.5,-1.75) -- (5,-2.25);
\draw [fill] (3.5,1.5) circle [x radius=0.05, y radius=0.1];
\draw [fill] (4,2.25) circle [x radius=0.05, y radius=0.1];
\draw [fill] (4.5,2.75) circle [x radius=0.05, y radius=0.1];
\draw [fill] (5,2.25) circle [x radius=0.05, y radius=0.1];
\draw [fill] (5.5,1.5) circle [x radius=0.05, y radius=0.1];
\draw [fill] (6,0) circle [x radius=0.05, y radius=0.1];
\draw [fill] (4,0.75) circle [x radius=0.05, y radius=0.1];
\draw [fill] (4.5,0.25) circle [x radius=0.05, y radius=0.1];
\draw [fill] (5,0.75) circle [x radius=0.05, y radius=0.1];
\draw [fill] (4.5,1.25) circle [x radius=0.05, y radius=0.1];
\draw [fill] (4.5,1.75) circle [x radius=0.05, y radius=0.1];
\draw [fill] (3.5,-1.5) circle [x radius=0.05, y radius=0.1];
\draw [fill] (4,-2.25) circle [x radius=0.05, y radius=0.1];
\draw [fill] (4.5,-2.75) circle [x radius=0.05, y radius=0.1];
\draw [fill] (5,-2.25) circle [x radius=0.05, y radius=0.1];
\draw [fill] (5.5,-1.5) circle [x radius=0.05, y radius=0.1];
\draw [fill] (4,-0.75) circle [x radius=0.05, y radius=0.1];
\draw [fill] (4.5,-0.25) circle [x radius=0.05, y radius=0.1];
\draw [fill] (5,-0.75) circle [x radius=0.05, y radius=0.1];
\draw [fill] (4.5,-1.25) circle [x radius=0.05, y radius=0.1];
\draw [fill] (4.5,-1.75) circle [x radius=0.05, y radius=0.1];
\draw (6,0) -- (6.5,3) -- (7,4.5) node[left] {$v_{14,4}$} -- (7.5,5.25) -- (8,5.75) node[above] {$v_{16,16}$} -- (8.5,5.25) -- (9,4.5) node[right] {$v_{18,4}$} -- (9.5,3) -- (10,0);
\draw (6.5,3) -- (7,1.5) -- (7.5,2.25) -- (8,2.75) -- (8.5,2.25) -- (9,1.5) -- (9.5,3);
\draw (7,1.5) -- (7.5,0.75) -- (8,1.25) -- (8.5,0.75) -- (9,1.5);
\draw (7.5,0.75) -- (8,0.25) -- (8.5,0.75);
\draw (7.5,2.25) -- (8,1.75) -- (8.5,2.25);
\draw (7,4.5) -- (7.5,3.75) -- (8,3.25) -- (8.5,3.75) -- (9,4.5);
\draw (7.5,3.75) -- (8,4.25) -- (8.5,3.75);
\draw (7.5,5.25) -- (8,4.75) -- (8.5,5.25);
\draw (6,0) -- (6.5,-3) -- (7,-4.5) node[left] {$v_{14,1}$} -- (7.5,-5.25) -- (8,-5.75) node[below] {$v_{16,1}$} -- (8.5,-5.25) -- (9,-4.5) node[right] {$v_{18,1}$} -- (9.5,-3) -- (10,0) node[below right] {$v_{20,1}$};
\draw (6.5,-3) -- (7,-1.5) -- (7.5,-2.25) -- (8,-2.75) -- (8.5,-2.25) -- (9,-1.5) -- (9.5,-3);
\draw (7,-1.5) -- (7.5,-0.75) -- (8,-1.25) -- (8.5,-0.75) -- (9,-1.5);
\draw (7.5,-0.75) -- (8,-0.25) -- (8.5,-0.75);
\draw (7.5,-2.25) -- (8,-1.75) -- (8.5,-2.25);
\draw (7,-4.5) -- (7.5,-3.75) -- (8,-3.25) -- (8.5,-3.75) -- (9,-4.5);
\draw (7.5,-3.75) -- (8,-4.25) -- (8.5,-3.75);
\draw (7.5,-5.25) -- (8,-4.75) -- (8.5,-5.25);
\draw [fill] (6.5,3) circle [x radius=0.05, y radius=0.1];
\draw [fill] (7,4.5) circle [x radius=0.05, y radius=0.1];
\draw [fill] (7.5,5.25) circle [x radius=0.05, y radius=0.1];
\draw [fill] (8,5.75) circle [x radius=0.05, y radius=0.1];
\draw [fill] (8.5,5.25) circle [x radius=0.05, y radius=0.1];
\draw [fill] (9,4.5) circle [x radius=0.05, y radius=0.1];
\draw [fill] (9.5,3) circle [x radius=0.05, y radius=0.1];
\draw [fill] (10,0) circle [x radius=0.05, y radius=0.1];
\draw [fill] (7,1.5) circle [x radius=0.05, y radius=0.1];
\draw [fill] (7.5,2.25) circle [x radius=0.05, y radius=0.1];
\draw [fill] (8,2.75) circle [x radius=0.05, y radius=0.1];
\draw [fill] (8.5,2.25) circle [x radius=0.05, y radius=0.1];
\draw [fill] (9,1.5) circle [x radius=0.05, y radius=0.1];
\draw [fill] (7.5,0.75) circle [x radius=0.05, y radius=0.1];
\draw [fill] (8,1.25) circle [x radius=0.05, y radius=0.1];
\draw [fill] (8.5,0.75) circle [x radius=0.05, y radius=0.1];
\draw [fill] (7.5,3.75) circle [x radius=0.05, y radius=0.1];
\draw [fill] (8,4.25) circle [x radius=0.05, y radius=0.1];
\draw [fill] (8.5,3.75) circle [x radius=0.05, y radius=0.1];
\draw [fill] (8,3.25) circle [x radius=0.05, y radius=0.1];
\draw [fill] (8,0.25) circle [x radius=0.05, y radius=0.1];
\draw [fill] (8,1.75) circle [x radius=0.05, y radius=0.1];
\draw [fill] (8,4.75) circle [x radius=0.05, y radius=0.1];

\draw [fill] (6.5,-3) circle [x radius=0.05, y radius=0.1];
\draw [fill] (7,-4.5) circle [x radius=0.05, y radius=0.1];
\draw [fill] (7.5,-5.25) circle [x radius=0.05, y radius=0.1];
\draw [fill] (8,-5.75) circle [x radius=0.05, y radius=0.1];
\draw [fill] (8.5,-5.25) circle [x radius=0.05, y radius=0.1];
\draw [fill] (9,-4.5) circle [x radius=0.05, y radius=0.1];
\draw [fill] (9.5,-3) circle [x radius=0.05, y radius=0.1];
\draw [fill] (7,-1.5) circle [x radius=0.05, y radius=0.1];
\draw [fill] (7.5,-2.25) circle [x radius=0.05, y radius=0.1];
\draw [fill] (8,-2.75) circle [x radius=0.05, y radius=0.1];
\draw [fill] (8.5,-2.25) circle [x radius=0.05, y radius=0.1];
\draw [fill] (9,-1.5) circle [x radius=0.05, y radius=0.1];
\draw [fill] (7.5,-0.75) circle [x radius=0.05, y radius=0.1];
\draw [fill] (8,-1.25) circle [x radius=0.05, y radius=0.1];
\draw [fill] (8.5,-0.75) circle [x radius=0.05, y radius=0.1];
\draw [fill] (7.5,-3.75) circle [x radius=0.05, y radius=0.1];
\draw [fill] (8,-4.25) circle [x radius=0.05, y radius=0.1];
\draw [fill] (8.5,-3.75) circle [x radius=0.05, y radius=0.1];
\draw [fill] (8,-3.25) circle [x radius=0.05, y radius=0.1];
\draw [fill] (8,-0.25) circle [x radius=0.05, y radius=0.1];
\draw [fill] (8,-1.75) circle [x radius=0.05, y radius=0.1];
\draw [fill] (8,-4.75) circle [x radius=0.05, y radius=0.1];

\draw [very thick][dotted] (10,0) -- (10.5,0);

\end{tikzpicture}
\caption{A graph with subexponential growth in which a single firefighter can contain any fire.}
\label{fig:2sn}
\end{figure}
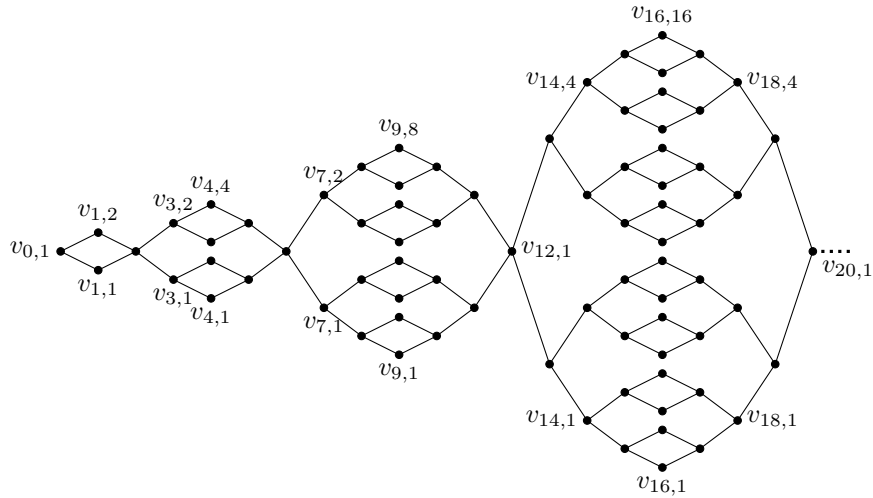

For a function $\beta\colon \N\to \N$ we denote by $\beta'\colon  \N\to \N$ the function given by $\beta'(0)=\beta(0)$ and $\beta'(n)=\beta(n)-\beta(n-1)$. Observe that in the case that $\beta$ is the growth function of a graph $G$ with respect to a vertex $g_0$, then $\beta'(n)$ is the cardinality of the sphere $S_G(n, g_0)$.  In particular, $\beta''(n)$ is the difference of cardinalities between the spheres $S_n$ and $S_{n-1}$.

\begin{theorem}\label{fImpliesContains}
Let $G$ be a locally finite connected graph with growth function $\beta\colon \N \to \N$. Suppose that $\beta''\colon \N\to \N$ is non-negative and non-decreasing. Then $G$ has the $\{f_n\}$-containment property with $f_n = 3 \beta''(2n)$.
\end{theorem}
\begin{proof}
Assume that the growth function is with respect to the vertex $g_0$.  Without loss of generality assume the initial fire is $B_G(x_0, n)$ where $n\in \mathbb{N}$ is arbitrary and $g_0$ is the root of $G$.
It is sufficient to show that  there exists $m>n$ such that $\sum_{k=1}^{m-n} \fffun{k} \geq \growthp{m}$, since then the fire can be contained by protecting all vertices of the sphere of radius $m$ about $g_0$  within $m-n$ turns.

Choose $m$ sufficiently large so that  $\sum_{k=1}^{m-n}\growthpp{2k} \geq \growthp{n}$ and  $m \geq 2n$. Then
\begin{equation}  2\sum_{k=1}^{n} \growthpp{2k} \geq \sum_{k=1}^{n} \left[ \growthpp{2k-1}+ \growthpp{2k} \right] \geq \sum_{k=n+1}^{2n} \growthpp{k} \end{equation}
and
$$2\sum_{k=n+1}^{m-n} \growthpp{2k} = 2\sum_{k=1}^{m-2n} \growthpp{2k+2n} \geq \sum_{k=1}^{m-2n} \growthpp{k+2n} = \sum_{k=2n+1}^{m} \growthpp{k}. $$
Therefore
$$2\sum_{k=1}^{m-n} \growthpp{2k} \geq  \sum_{k=n+1}^{2n} \growthpp{k} + \sum_{k=2n+1}^{m} \growthpp{k} = \sum_{k=n+1}^{m} \growthpp{k} .$$
Now we have
$$\sum_{k=1}^{m-n} \fffun{k} = 3\sum_{k=1}^{m-n} \growthpp{2k} \geq  \growthp{n} + \sum_{k=n+1}^{m} \growthpp{k} =  \growthp{m}$$
as required.
\end{proof}

\section{Homogeneous Growth}\label{sec:homogeneous}

\begin{definition}
Let $G$ be a  graph, let $g_0$ be a vertex of $G$. Let $S_n$ denote the sphere
of radius $n$ about $g_0$, and let $s_n$ denote the cardinality of $S_n$. For
subset $T$ of  $S_n$, let $T^*$ denote the subset of vertices of $S_{n+1}$ which are adjacent to at least one vertex of $T$. The growth of $G$ is \emph{homogeneous with respect to $g_0$} if there exists  $r\geq 0$ such that for any $n\geq r$ and  any non-empty subset $T$ of  $S_n$,
\begin{equation} \frac{| T^*|}{ |T| }   \geq \frac{s_{n+1}}{s_n} \geq 1.\end{equation}
 Under these conditions, we say that $G$ has \emph{homogeneous growth with respect to $g_0$ from radius $r$.}
\end{definition}

\begin{theorem}\label{thm:hgrowth}
Let $G$ be a graph which has homogeneous growth with respect to $g_0$, and let $\{f_n\}$ be a non-decreasing positive sequence. If $G$ has the $\{f_n\}$-containment property, then the series $ \sum_{n=1}^\infty \frac{f_n}{s_n}$ diverges.
\end{theorem}

\begin{proof}
Suppose that $G$ has homogeneous growth with respect to $g_0$ from radius $r$, and  it has the $\{f_n\}$-containment property.   The homogeneous hypothesis implies that $G$ is an infinite graph.

Let $m$ be an arbitrary integer greater than $r$.  Let $\{W_k\colon k\geq 1\}$ be an $\{f_n\}$-containment strategy for the initial fire $X_0=B(g_0, m)$, let $X_k$ denote the set of vertices on fire at time $k$, and suppose that the fire is contained after time $\ell$, i.e., $X_k=X_{\ell-1}$ for every $k \geq \ell$ and $\ell$ is minimal with this property. Without loss of generality, assume that $W_i=\emptyset$ for $i> \ell$. Since $G$ is an infinite graph, $\ell>0$.

 Let $S_k$ denote  the collection of vertices at distance $k$ from $g_0$. For $k\geq 0$, let
\begin{equation} T_k = S_{m+k}\cap X_k,\end{equation}
 equivalently, $T_k$ is the set of vertices on fire in the sphere $S_{m+k}$ at time $k$. In particular,  $T_{k+1}$ is the set of vertices in $T_{k}^*=B_G(T_k, 1)\cap S_{m+k+1}$ which remain unprotected at time $k$.

From here on, the cardinalities of $S_k$ and $T_k$ and $T_k^*$  are denoted by $s_k$, $t_k$ and $t_{k+1}^*$ respectively (the shift of indexes is done on purpose).  Let
\begin{equation}p_k=t_{m+k}^*-t_{m+k}.\end{equation}
Observe that $p_k$ is \emph{at most} the number of vertices of the sphere $S_{m+k}$ that were protected up to time $k$.

Observe that the total number of vertices protected up to time $q$  is bounded above by $\sum_{k=1}^q f_k$, therefore
\begin{equation} \label{eq:sfigeqspi} \sum_{k=1}^q f_k \geq \sum_{k=1}^q  p_k\end{equation}
for every positive integer $q$.

The homogeneous growth assumption implies that
\begin{equation} \frac{t_{m+k}^*}{t_{m+k-1}} \geq \frac{s_{m+k}}{s_{m+k-1}}   \end{equation}
for every $k\geq 0$.
It follows that
\begin{equation} t_{m+k} = t^*_{m+k}- p_k \geq \frac{s_{m+k}}{s_{m+k-1}}t_{m+k-1} -   p_k\end{equation}
for every $k\geq 0$.
Therefore
\begin{equation} \label{eq:hgrowth3}\begin{split}
 & t_{m+\ell}^* \geq \frac{s_{m+\ell}}{s_{m+\ell-1}}t_{m+\ell-1} \\
& \geq   \frac{s_{m+\ell}}{s_{m+\ell-2}} t_{m+\ell-2} - \frac{s_{m+\ell}}{s_{m+\ell-1}} p_{\ell -1}    \\
& \geq     \frac{s_{m+\ell}}{s_{m+\ell-3}} t_{m+\ell-3} -  \frac{s_{m+\ell}}{s_{m+\ell-2}}  p_{ \ell-2}  - \frac{s_{m+\ell}}{s_{m+\ell-1}} p_{\ell -1}
  \\
&\quad \quad \vdots \\
& \geq \frac{s_{m+\ell}}{s_{m}} t_{m}  - \frac{s_{m+\ell}}{s_{m+1}} p_{1}  - \frac{s_{m+\ell}}{s_{m+2}} p_2  -   \ldots -  \frac{s_{m+\ell}}{s_{m+\ell-2}}  p_{ \ell-2}  - \frac{s_{m+\ell}}{s_{m+\ell-1}}   p_{\ell -1}.
\end{split}
\end{equation}

Since $p_\ell=  t_{m+\ell}^*$, re-arranging~\eqref{eq:hgrowth3} yields
\begin{equation} \label{eq:hgrowth4} \sum_{k=1}^{\ell} \frac{p_k}{s_{m+k}}   \geq \frac{t_m}{s_m}  =1 \end{equation}
The homogeneous hypothesis guarantees that the sequence $\{s_n\}$ is non-decreasing. Considering the inequality~\eqref{eq:sfigeqspi} and applying Lemma~\ref{lem:easyineq} to the sequences $\{f_k\}_{k=1}^\ell$, $\{p_k\}_{k=1}^\ell$ and $\{s_{m+k}\}_{k=1}^\ell$,  we obtain
\begin{equation} \label{eq:hgrowth5} \sum_{k=1}^{\ell} \frac{f_k}{s_{m+k}} \geq \sum_{k=1}^{\ell} \frac{p_k}{s_{m+k}}. \end{equation}
Since the sequence $\{f_n\}$ is non-decreasing, it follows that
\begin{equation}\label{eq:hgrowth6} \sum_{k=1}^{\ell} \frac{f_{m+k}}{s_{m+k}} \geq \sum_{k=1}^{\ell} \frac{f_k}{s_{m+k}}. \end{equation}
Inequalities~\eqref{eq:hgrowth4}~\eqref{eq:hgrowth5} and~\eqref{eq:hgrowth6} yield
\begin{equation}\label{eq:hgrowth7} \sum_{k=1}^{\ell} \frac{f_{m+k}}{s_{m+k}} \geq 1 .\end{equation}
Since $m$ was arbitrary and $\ell\geq 1$, inequality~\eqref{eq:hgrowth7} implies that the series $\sum_{k=1}^\infty \frac{f_k}{s_k}$  diverges.
\end{proof}

\begin{example}
Let $G$ be the infinite regular tree on which every vertex has degree $\delta+1$. Then $G$ has homogeneous growth since
\begin{equation} \frac{ \left |T^* \right|}{ |T|} = \frac{s_{n+1}}{s_n} = \delta.\end{equation}
\end{example}

\begin{definition}\label{def:Ld}
The $d$-dimensional square grid $\mathbb{L}^d$ is  the graph with vertex set
\begin{equation}V(\mathbb{L}^d)= \{(x_1,x_2, \ldots, x_d) \mid x_i \in \mathbb{Z} \}\end{equation}
and edge set
\begin{equation}E(\mathbb{L}^d)= \left\{(x,y)\in V\times V \;\middle|\; \sum_{i=1}^d |x_i-y_i| =1\right\}.\end{equation}
The \emph{positive orthant of $\mathbb L^d$} is the subgraph $\mathbb L^d_+$   expanded by the collection of vertices $(x_1,x_2, \ldots, x_d) $ for which $x_i\geq 0$ for every $i$. The vertex of $\mathbb L^d_+$  whose coordinates are all zero is called \emph{the origin}.
\end{definition}

\begin{remark}\label{rem:snLd}
For the {positive orthant $\mathbb L^d_+$}, the number of vertices at distance $m$ from the origin is given by ${m+d-1 \choose d-1}$ which is a polynomial of degree $d-1$ in $m$.  This is a standard counting argument since any  vertex $(x_1,x_2, \ldots, x_d)\in \mathbb{L}^d_+$ is a  sum $\sum_{i=1}^d x_i e_i$ and its distance to the origin is exactly the sum $\sum_{i=1}^d x_i$.  As consequence, the growth function of $\mathbb L^d_+$ (and hence of $\mathbb L^d$) with respect to the origin is  bounded from above by a polynomial of degree $d$.
\end{remark}

\begin{proposition}\label{prop:Ldsatisfies}
The graph $\mathbb{L}^d_+$ has homogeneous growth.
\end{proposition}

\begin{proof}  Let $T$ be a subset of the sphere $S_m \subset \mathbb L^d_+$ of radius $m$ centered at the origin. Denote by $N(T)$ the set $B_{\mathbb L^d_0} (T, 1)\cap S_{m+1}$.  Let $e_i$ be the vertex of $\mathbb L^d_0$ whose $i$-th coordinate is one, and any other coordinate is zero. Any   vertex $(x_1,x_2, \ldots, x_d)\in \mathbb{L}^d_+$ is a  sum $\sum_{i=1}^d x_i e_i$.  From the remark above, we have that
\begin{equation} \frac{s_{m+1}}{s_m} = 1+\frac{d-1}{m+1}.\end{equation}
Let $H_{i}=T+e_i$ and observe that $|H_i|=|T|$ and $N(T)= \bigcup_{i=1}^d H_i$.
Observe that verifying
\begin{equation} \left|\bigcup_{i=1}^d H_i \right| \geq \left(1+\frac{d-1}{m+1}\right) |H_1|,\end{equation}
proves that $\mathbb{L}^d_+$ has homogeneous growth. Therefore, it is sufficient to prove that \begin{equation} (m+1)\left|\bigcup_{i=2}^d H_i \cap H_1^c\right| \geq (d-1)|H_1|.\end{equation}
Letting $P_1 = H_1 \times \{e_2, \ldots e_d\}$, it is clear that $|P_1| = (d-1)|H_1|$.   Additionally consider $P_2 \subset \bigcup_{i=2}^d H_i \cap H_1^c \times \{e_2, \ldots, e_d\} \times \{1, \ldots, m+1\}$, where $(y,e_j,k) \in P_2$ only if $y-ke_j+ke_1 \in S_{m+1}$.  Now since $y=\sum_{i=1}^d y_i e_i\in S_{m+1}$ we have $\sum_{i=2}^d y_i \leq m+1$, hence there are at most $m+1$ elements  $(y,e_j,k) \in P_2$ for each $y$, which implies $|P_2|\leq (m+1)|\bigcup_{i=2}^d H_i \cap H_1^c|$.
Now for any $(x, e_i)\in P_1$ we can let $b=\min\{k \mid x-ke_1+ke_i \notin H_1\}$, noting that since $x-(b-1)e_1+(b-1)e_i\in H_1$ we have $x-be_1+be_i \in H_i$.  Then $(x-be_1+be_i, e_i, b) \in P_2$ as $x-be_1+be_i \in H_i \setminus H_1$  and $x-be_1+be_i-be_i+be_1=x \in H_1 \subset S_{m+1}$.  We have constructed a map from $P_1$ to $P_2$,  $$(x,e_i) \mapsto (x-be_1+be_i, e_i, b).$$ As no two distinct elements of $P_1$ can be mapped to the same element of $P_2$, this map is an injection and therefore $|P_1|\leq |P_2|$, demonstrating the inequality.
\end{proof}

\begin{corollary}\label{cor:conj} Let $d$ and $q$ be positive integers.
If $\lim_{n\to \infty} \frac{n^q}{n^{d-2}}=0$, then $\mathbb{L}^d$ does not satisfy the $O(n^q)$-containment property. In particular, $\mathbb{L}^d$ does not satisfy the $O(n^{d-3})$-containment property.
\end{corollary}
\begin{proof} 
Since containment properties are subgraph-hereditary, Proposition~\ref{prop:subgraph},  it is enough to argue for the orthant $\mathbb{L}^d_+$.
The limit assumption implies that $d-2-q\geq 1$ and therefore the series $\sum \frac{n^q}{n^{d-1}}$ converges.  For $\mathbb{L}^d_+$, the sequence $s_n$ is a polynomial of degree $d-1$ in $n$, see Remark~\ref{rem:snLd}.
Since $\mathbb{L}^d_+$ has homogeneous growth, Proposition~\ref{prop:Ldsatisfies},   we have that  Theorem~\ref{thmi:homogeneous} implies that $\mathbb{L}^d_+$  does not satisfy the $O(n^q)$-containment property.   \end{proof}

\section{Quasi-isometry invariance of the Containment Property}\label{sec:qi}

 Definition~\ref{def:qi} can be re-stated as described in the following remark.

\begin{remark} \cite[Page 138, Exercise 8.16]{BrHa99} \label{rem:qidef} The graphs $G$ and $H$  are  quasi-isometric if there is an  integer $c\geq 1$ and functions functions between their vertex sets
$\phi \colon G \to H$ and $\psi \colon H \to G$ such that  for every pair of vertices $g_1, g_2 \in G$,
\begin{equation}   \dist_H (\phi g_1, \phi g_2 )  \leq c\dist_G (g_1, g_2)  + c,\end{equation}
for every pair of vertices $h_1,h_2 \in H$
\begin{equation}  \dist_G (\psi h_1, \psi h_2 )  \leq c \dist_H (h_1, h_2) + c,\end{equation}
 for every $g\in G$
\begin{equation} \dist_G (g, \psi \phi g)\leq c,\end{equation}
and for every $h \in H$
\begin{equation} \dist_H (h,  \phi \psi h)\leq c.\end{equation}
In this case, the pair $(\phi, \psi)$ is called a $c$-quasi-isometry from $G$ to $H$.
\end{remark}

\begin{definition}
Given  non-decreasing sequences $f\colon \N\to \N$ and $g\colon \N\to \N$. The relation $f \preccurlyeq g$ is defined as the existence of an integer $c>0$ such that $f(n) \leq cg(cn+c)+c$ for every $n\geq c$. If $f \preccurlyeq g \preccurlyeq f$ then we say that $f$ and $g$ have \emph{equivalent asymptotic growth} and write $f \sim g$.
\end{definition}

\begin{remark}
Suppose that $\{a_n\}$ is asymptotic equivalent to $\{b_n\}$. If $\{a_n\}$ is $O(n^d)$,   then $\{b_n\}$ is $O(n^d)$.
\end{remark}

\begin{theorem} \label{thm:qi}
Let $G$ and $H$ be graphs with bounded degree. Suppose that $G$ is quasi-isometric to $H$, and let $\{f_n\}$ be a non-decreasing sequence of integers.
If $G$ satisfies the $\{f_n\}$-containment property then  $H$ satisfies the $\{b_n\}$-containment property where $\{b_n\}$ is a sequence with asymptotic growth equivalent to  $\{f_n\}$.
\end{theorem}

The basic strategy of the proof is to start with an arbitrary fire on $H$ and then construct a corresponding fire on $G$.  Then the containment strategy in $G$ is transformed into a containment strategy in $H$ using the quasi-isometry between the graphs. A description of this process is illustrated  in Figure~\ref{fig:quasicontain}.

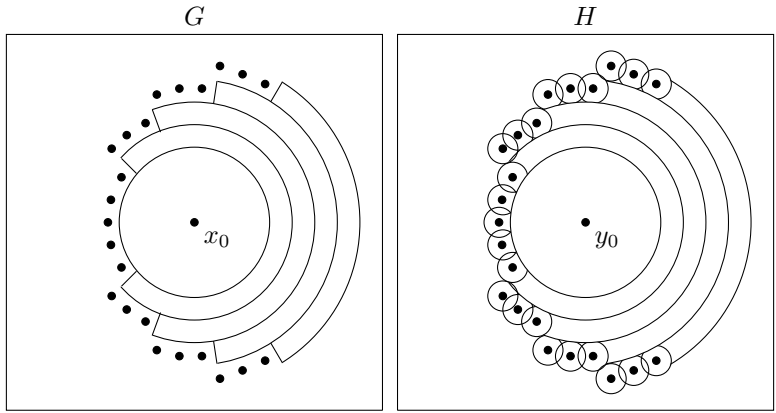
\begin{figure}
\center

\begin{tikzpicture}

\draw (0.1,0)--(5.1, 0)--(5.1,5)--(0.1,5)--(0.1,0);
\draw (-0.1,0)--(-5.1, 0)--(-5.1,5)--(-0.1,5)--(-0.1,0);

\draw (2.6,2.5) [fill=white] circle [radius=2.2];
\draw (-2.6,2.5) [fill=white] circle [radius=2.2];
\draw [white,fill=white] (0.3,0.2) rectangle (3.72,4.8);
\draw [white,fill=white] (-4.9,0.2) rectangle (-1.44,4.8);
\draw (-1.43, 0.63) -- (-1.6, 0.92);
\draw (-1.43, 4.37) -- (-1.6, 4.08);

\draw (-2.26,4.58) [fill] circle [radius=0.05];
\draw (-2.26,.42) [fill] circle [radius=0.05];
\draw (3.24,4.47) circle [radius=0.2];
\draw (3.24,4.47) [fill] circle [radius=0.05];
\draw (3.54,4.34) circle [radius=0.2];
\draw (3.54,4.34) [fill] circle [radius=0.05];
\draw (3.24,.53) circle [radius=0.2];
\draw (3.24,.53) [fill] circle [radius=0.05];
\draw (3.54,.66) circle [radius=0.2];
\draw (3.54,.66) [fill] circle [radius=0.05];
\draw (-1.66,.66) [fill] circle [radius=0.05];
\draw (-1.66,4.34) [fill] circle [radius=0.05];
\draw (-1.96,.53) [fill] circle [radius=0.05];
\draw (-1.96,4.47) [fill] circle [radius=0.05];

\draw (2.6,2.5) [fill=white] circle [radius=1.9];
\draw (-2.6,2.5) [fill=white] circle [radius=1.9];
\draw [white,fill=white] (0.4,0.5) rectangle (2.87,4.4);
\draw [white,fill=white] (-4.8,0.5) rectangle (-2.3,4.5);
\draw (-2.3, 0.62) -- (-2.35, 0.95);
\draw (-2.3, 4.38) -- (-2.35, 4.05);

\draw (2.94,4.58) [fill] circle [radius=0.05];
\draw (2.94,4.58) circle [radius=0.2];
\draw (2.94,.42) [fill] circle [radius=0.05];
\draw (2.94,.42) circle [radius=0.2];

\draw [fill] (-3.1,4.2) circle [radius=0.05];
\draw [fill] (-3.1,0.8) circle [radius=0.05];
\draw [fill] (-2.8,4.28) circle [radius=0.05];
\draw [fill] (-2.8,0.72) circle [radius=0.05];
\draw [fill] (-2.5,4.28) circle [radius=0.05];
\draw [fill] (-2.5,0.72) circle [radius=0.05];
\draw [fill] (2.1,4.2) circle [radius=0.05];
\draw [fill] (2.1,0.8) circle [radius=0.05];
\draw [fill] (2.4,4.28) circle [radius=0.05];
\draw [fill] (2.4,0.72) circle [radius=0.05];
\draw [fill] (2.7,4.28) circle [radius=0.05];
\draw [fill] (2.7,0.72) circle [radius=0.05];
\draw (2.1,4.2) circle [radius=0.2];
\draw (2.1,0.8) circle [radius=0.2];
\draw (2.4,4.28) circle [radius=0.2];
\draw (2.4,0.72) circle [radius=0.2];
\draw (2.7,4.28) circle [radius=0.2];
\draw (2.7,0.72) circle [radius=0.2];

\draw (2.6,2.5) [fill=white] circle [radius=1.6];
\draw (-2.6,2.5) [fill=white] circle [radius=1.6];
\draw [white,fill=white] (0.8,1) rectangle (2.03,4);
\draw [white,fill=white] (-4.4,1) rectangle (-3.17,4);
\draw (-3.16, 4)--(-3.05, 3.70);
\draw (-3.16, 1)--(-3.05, 1.30);

\draw [fill] (-3.7,3.48) circle [radius=0.05];
\draw [fill] (-3.7,1.52) circle [radius=0.05];
\draw [fill] (1.5,3.48) circle [radius=0.05];
\draw [fill] (1.5,1.52) circle [radius=0.05];
\draw [fill] (-3.5,3.66) circle [radius=0.05];
\draw [fill] (-3.5,1.34) circle [radius=0.05];
\draw [fill] (1.7,3.66) circle [radius=0.05];
\draw [fill] (1.7,1.34) circle [radius=0.05];
\draw [fill] (-3.25,3.82) circle [radius=0.05];
\draw [fill] (-3.25,1.18) circle [radius=0.05];
\draw [fill] (1.95,3.82) circle [radius=0.05];
\draw [fill] (1.95,1.18) circle [radius=0.05];
\draw (1.7,3.66) circle [radius=0.2];
\draw (1.7,1.34) circle [radius=0.2];
\draw (1.95,3.82) circle [radius=0.2];
\draw (1.95,1.18) circle [radius=0.2];

\draw (2.6,2.5) [fill=white] circle [radius=1.3];
\draw (-2.6,2.5)  circle [radius=1.3];
\draw [white,fill=white] (1.25,1.7) rectangle (1.6,3.3);
\draw [white,fill=white] (-4,1.65) rectangle (-3.25,3.35);
\draw (-3.57,3.35) -- (-3.37, 3.15);
\draw (-3.57,1.65) -- (-3.37, 1.85);
\draw (1.5,3.48) circle [radius=0.2];
\draw (1.5,1.52) circle [radius=0.2];

\draw (2.6,5) node[above] {$H$};
\draw (-2.6,5) node[above] {$G$};
\draw [fill] (-3.75,2.5) circle [radius=0.05];
\draw [fill] (-3.71,2.2) circle [radius=0.05];
\draw [fill] (-3.71,2.8) circle [radius=0.05];
\draw [fill] (-3.57,3.1) circle [radius=0.05];
\draw [fill] (-3.57,1.9) circle [radius=0.05];
\draw [fill] (1.45,2.5) circle [radius=0.05];
\draw [fill] (1.49,2.2) circle [radius=0.05];
\draw [fill] (1.49,2.8) circle [radius=0.05];
\draw [fill] (1.63,3.1) circle [radius=0.05];
\draw [fill] (1.63,1.9) circle [radius=0.05];
\draw (1.45,2.5) circle [radius=0.2];
\draw (1.5,2.2) circle [radius=0.2];
\draw (1.5,2.8) circle [radius=0.2];
\draw (1.63,1.9) circle [radius=0.2];
\draw (1.63,3.1) circle [radius=0.2];

\draw (2.6,2.5) [fill=white] circle [radius=1];
\draw (-2.6,2.5) [fill=white] circle [radius=1];
\draw [fill] (2.6,2.5) node[below right] {$y_0$} circle [radius=0.05];
\draw [fill] (-2.6,2.5) node[below right] {$x_0$} circle [radius=0.05];

\end{tikzpicture}
\caption{A containment strategy on $G$ and translated to $H$.}
\label{fig:quasicontain}
\end{figure}

\subsection{Proof of Theorem~\ref{thm:qi}}

Let $G$ and $H$ be graphs with the property that every vertex has degree bounded above by $\delta>0$. Let $(\phi, \psi)$ be a $c$-quasi-isometry from $G$ to $H$ as described in Remark~\ref{rem:qidef}.

In view of Proposition~\ref{scalingTurns}, we  assume that $G$ has the $(\{a_n\},2c)$-containment property where $a_{n+1}=\sum_{i=1}^{2c} f_{2cn+i}$.  Since $f_n$ is a non-decreasing sequence,
\begin{equation}
 f_{2cn+1} \leq a_{n+1} \leq 2c f_{2c(n+1)}
\end{equation}
and hence  the sequences $a_n$ and $f_n$ have equivalent asymptotic growth.

We prove below that  $H$ has the $(\{b_n\}, 1)$-containment property where
\begin{equation} b_n = a_n\cdot \delta^{r+1}  , \quad \quad r=c^2+2c .\end{equation}
It is immediate that the sequences $a_n$ and $b_n$ have equivalent asymptotic growth,  the same equivalence holds for the sequences $f_n$ and $b_n$.

Given a vertex $h_0\in H$ and $q\geq 0$, a $(\{b_n\},1)$-containment strategy $\{Q_k\colon k\geq 1\}$ for \begin{equation}Y_0=B_H(h_0, q) \end{equation} is obtained as follows.
Let $g_0=\psi h_0$, and  let $W_1, W_2, \ldots$ be a $(\{a_n\}, 2c)$-containment strategy for \begin{equation} X_0=B_G\left(g_0, 2c(q+2)\right) .\end{equation}
For $k\geq 1$, define recursively $Q_k$ and $Y_k$ as
\begin{equation}Q_k = \bigcup_{ g \in W_k}   B_H (\phi g, r)  \setminus Y_{k-1}\end{equation}
and
\begin{equation}Y_{k+1}=  B_H(Y_k, 1) \setminus Q_{k+1} .\end{equation}

The proof of Theorem~\ref{thm:qi} reduces to prove that $\{Q_k\colon k\geq 1\}$ is a $(\{b_n\},1)$-containment strategy for the initial fire $Y_0$. This entails verifying the three statements of Definition~\ref{def:firegame2}.  The first two statements are immediate:

\subsubsection{} For every $k$, the set $Q_k$ has cardinality at most $b_k$. Indeed, since each $W_k$ has cardinality bounded by $a_n$, and a ball in $H$ of radius $r$ about any vertex has cardinality bounded by $\delta^{r+1}$, it follows that   $|Q_k|\leq a_n\cdot \delta^{r+1}=b_n$, proving the first statement.

\subsubsection{} The sets $Y_k$ and $Q_{k+1}$ are disjoint by construction. Moreover, an easy induction argument shows that a vertex $h$ belongs to $Y_k$ if and only if there is a path $h_0, h_1, h_2, \ldots , h_\ell=h$ such that no $h_i$ is in $Q_1\cup \cdots \cup Q_k$, and $\ell \leq q+k$.

\subsubsection{} It is left to verify the third statement of Definition~\ref{def:firegame2} for  $\{Q_k\colon k\geq 1\}$. This part is where most of the work in the proof is.

Let $X_{k}$  consist  of the vertices of $G$ which are on fire at time $k$ given the initial fire $X_0$ and the strategy $\{W_n \colon n\geq 1\}$. Equivalently, $X_k$ consists of the vertices of $G$ that
 are connected to a vertex in $X_{k-1}$ by a path of length at most $2c$ containing no vertices in $W_1\cup \cdots \cup W_k$.

 Define
\begin{equation} r_k = 2c(q+k+2).\end{equation}
Observe that  $X_k$ consists of vertices $g \in G$ such that there is a path from $g_0$ to $g$ of length at most $r_k$ that does not contain vertices in $W_1\cup \cdots \cup W_k$. In particular,
\begin{equation} X_k \subseteq B_G(g_0, r_k)\end{equation}
for  all $k\geq 0$.

\begin{lemma}\label{lem:qi}
For every $k$, if $h \in Y_k$ then $\psi h \in X_{k-1}$.
\end{lemma}
\begin{proof}[Proof of Lemma~\ref{lem:qi}]
We argue by induction on $k$.
If $h \in Y_1$  then $\dist_H (h_0, h)\leq q+1$ and hence
\begin{equation}\dist_G(g_0, \psi h) \leq c(q+1)+c \leq 2c(q+2).\end{equation}
It follows that $\psi h$ belongs to $X_0=B_G(x_0, r_0)$.

Assume inductively that $h \in Y_j$ implies $\psi h \in X_{j-1}$ for all $j<k$ and $2\leq k$.
Suppose $h\in Y_k$. Then there exists a path
\begin{equation}h_0, h_1, h_2, \ldots , h_\ell=h\end{equation}
such that $\ell \leq q+k$ and no $h_i$ is in $Q_1\cup \cdots \cup Q_k$ (since it is a path, $\dist_H (h_i, h_{i+1})=1$ for $i<\ell$). Consider the sequence of vertices
\begin{equation}\psi h_0, \psi h_1, \psi h_2, \ldots , \psi h_\ell.\end{equation}
Since $\dist_G (\psi h_{i-1} , \psi h_i)\leq  c\dist_H (h_i, h_{i+1}) + c =2c$,  there is a path $\gamma_i$ of length at most $2c$ from $\psi h_{i-1}$ to $\psi h_i$. Consider the path $\gamma$ from $\psi h_0$ to $\psi h$ resulting from the concatenation $\gamma_1 \cdots \gamma_\ell$. Observe that the length of $\gamma$ is at most $2c\ell \leq 2c(q+k) \leq r_{k-1}$. We prove below that  that no vertex of $\gamma$ is in the set $W_1\cup \cdots \cup W_{k-1}$, which implies that $\psi h \in X_{k-1}$ completing the proof.

Suppose there are vertices of $\gamma$ in $W_1\cup \cdots \cup W_{k-1}$. By construction,  each vertex of $\gamma$  is at distance at most $c$ from a vertex  of the form $\psi h_i \in \gamma$. It follows that we can choose a vertex $g$ of $\gamma$  and a vertex of the form $\psi h_j$  of $\gamma$ (they might be the same vertex) with the following  properties:  first \begin{equation}g \in  W_1\cup \cdots \cup W_{k-1}\end{equation} and second, the subpath of $\gamma$ between $g$ and $\psi h_j$ has length at most $c$ and it has only one vertex in  $W_1\cup \cdots \cup W_{k-1}$, namely $g$.
{Let \begin{equation}t\leq k-1\end{equation} be the least integer such that \begin{equation}g\in W_t.\end{equation}} Since
\begin{equation}\dist_G(\phi g, h_j) \leq  \dist_H ( \phi g, \phi \psi h_j) + \dist_H (\phi \psi h_j, h_j) \leq c^2+2c  = r,\end{equation}
it follows  that either $h_j \in Q_t$ or $h_j \in Y_{t-1}$. The former case is impossible by the assumption on the path from $h_0$ to $h$. Therefore $h_j \in Y_{t-1}$ and then the induction hypothesis implies that $\psi h_j \in X_{t-2}$.
Since the subpath of $\gamma$ between $\psi h_j$ and $g$ has no vertices in $W_1\cup \cdots \cup W_{t-1}$ and $\psi h_j \in X_{t-2}$, it follows that $g\in X_{t-1}$. This implies that {\begin{equation}g\notin W_t\end{equation}} which is a contradiction. This completes the proof of the lemma.
\end{proof}

Since  $\{W_k\colon k\geq 1\}$ is an $(\{a_n\},2c)$-containment strategy for $X_0$ in $G$, there is $M\geq 0$ such that $X_n = X_M$  for every $n\geq M$.  Then Lemma~\ref{lem:qi}  implies that  $\psi Y_n \subseteq X_M \subseteq B_G(g_0, r_M)$ for $n>M$.  It follows that
\begin{equation} Y_n \subseteq B_H(\phi \psi Y_n, c) \subseteq \phi B_G(g_0, r_M+c) \subseteq B_H(h_0, cr_M+c^2+c)\end{equation}
for every $n$. Since $H$ is a locally finite graph, Proposition~\ref{prop:locfin} implies that there is $N\geq 0$ such that
\begin{equation} Y_n =Y_N \end{equation}
for every $n\geq N$.

We have shown that $\{Q_k\colon k\geq 1\}$ is a $(\{b_n\},1)$-containment strategy for $B_H(h_0, q)$ in $H$. Since $h_0\in H$  and $q\geq 0$ were arbitrary, Remark~\ref{rem:reduction} implies that $H$ has the $(\{b_n\},1)$-containment property concluding the proof of Theorem~\ref{thm:qi}. \hfill \qedsymbol

\begin{remark}[Remark on the proof of Theorem~\ref{thm:qi} and the time to contain fires]
When a graph satisfies a containment property, there is the associated measure of how long it takes to efficiently contain fires.  
Specifically, suppose that $G$ is connected and locally finite graph that has the $(\{f_n\}, r)$-containment property, and let $g_0$ be a vertex of $G$. 
For $k\geq 1$, let $T_k=T_G(\{f_n\}, r, k)$ be the minimal integer so that there is an $\{f_n\}$-containment strategy for $X_0=B_G (g_0, k)$ such that  $X_n \subset B_G(g_0, r T_k)$ for every $n$.  The sequence $\{T_k\}$ is called the \emph{time $(\{f_n\}, r)$-containment sequence for $G$}.  The proof of Proposition~\ref{scalingTurns} shows that the sequences $T_G(\{f_n\}, r, k)$ and $T_G(\{\ceil{f_n/r} \}, 1, k)$ have equivalent asymptotic growth. Analogously, the asymptotic growth class of $T_G(\{f_n\}, r, k)$ is independent of chosen vertex $g_0$.

Therefore, under the assumptions of Theorem~\ref{thm:qi}, the time containment sequences $\{T_G(\{f_n\}, 1, k)\}$ and $\{T_H(\{b_n\}, 1, k)\}$ have equivalent asymptotic growth.

Time containment sequences are not necessarily linear. We believe that this time complexity plays a role in addressing problems like Question~\ref{que:mohar} or Question~\ref{que:ggt}. 

\end{remark}

\subsection{Graphs quasi-isometric to trees}

\begin{corollary}\label{cor:tree}
If a graph $H$  contains a subgraph  quasi-isometric to the infinite $\delta$-regular tree with $\delta\geq 3$, then $H$ does not satisfy a polynomial containment property.
\end{corollary}
\begin{proof}
Let $G$ be an infinite tree on which every vertex has degree $\delta+1 \geq 3$. Then $G$ has homogeneous growth, and the size $s_n$ of a sphere of radius $n$ is bounded from below by the exponential function $\delta^n$. Therefore, if $\{f_n\}$ is a sequence of type $O(n^d)$ then $\sum \frac{f_n}{s_n}$ converges. By  Theorem~\ref{thmi:homogeneous}, the tree $G$ does not satisfy a polynomial containment property.
By Theorem~\ref{thm:qii} any graph quasi-isometric to $G$ does not satisfy a polynomial containment property.  Since containment properties are subgraph-hereditary, Proposition~\ref{prop:subgraph}, the statement follows.
\end{proof}

\section{Containment property for groups} \label{sec:groups}

\subsection{The \v Svark-Milnor Lemma}

Let $\Gamma$ be a group acting by isometries on a metric space $X$. We will only consider the case that $X$ is the Euclidean plane $\mathbb E^2$ or the hyperbolic plane $\mathbb H^2$.  The action is \emph{proper}   if for each $x\in X$ there is $r>0$ such that the set $\{\gamma \in \Gamma \colon \gamma B_X(x, r) \cap B_X(x, r) \neq \emptyset\}$ is finite. The action is \emph{cocompact} if there exists a compact set $K\subseteq X$ such that $X=\bigcup_{\gamma \in \Gamma} \gamma K$.

Any finitely generated group can be considered as a metric space. Specifically, let $\Gamma$ be a finitely generated group with a finite generating set $\mathcal A\subset \Gamma$. This finite generating set  induces a metric on $\Gamma$ known as the \emph{word metric} where   $\dist_{\mathcal A} (\gamma_1, \gamma_2)$ is defined as the length of the shortest word in the generators $\mathcal A$ representing the element of the group $\gamma_2\gamma_1^{-1}$. The word metrics associated with two different finite generating sets are quasi-isometric.

\begin{proposition}[\v Svark-Milnor Lemma]\cite[Page 140]{BrHa99}
Let $X$ be a geodesic metric space. If $\Gamma$ acts properly and cocompactly  by isometries  on $X$, then $\Gamma$ is finitely generated, and for any choice of basepoint $x_0\in X$, the map $\Gamma \to X$ given by $\gamma \mapsto \gamma x_0$ is a quasi-isometry.
\end{proposition}

\subsection{The containment property for groups}

 Let $\Gamma$ be a finitely generated group, and let $\mathcal A$ be a finite generating set. The Cayley graph $\mathcal C (\Gamma, \mathcal A)$ is the directed graph with vertex set $\Gamma$, and edge set corresponding to pairs $(\gamma, \gamma \alpha)$ where $\gamma\in \Gamma$ and $\alpha \in \mathcal A$. We consider Cayley graphs as undirected graphs by ignoring the orientation of edges. Observe that in this setting, every vertex of $\mathcal C (\Gamma, \mathcal A)$ has degree $2|\mathcal A|$.

For any finitely generated group, the Cayley graphs  associated with two different finite generating sets  are quasi-isometric. This can be seen as a consequence of the \v Svark-Milnor lemma by considering Cayley graphs as geodesic metric spaces (let each edge have length one and consider the induced path-metric) and observing that $\Gamma$ acts by isometries, properly and cocompactly on its Cayley graphs.

This fact  implies that invariants of (bounded degree) graphs which are preserved under quasi-isometry become invariants of finitely generated groups.  In particular, we say that the finitely generated groups $\Gamma$ and $\Delta$ are quasi-isometric if they have quasi-isometric Cayley graphs.  A property of groups is said to be \emph{geometric} if it is preserved under quasi-isometry.  The study of geometric properties of finitely generated groups was started by Gromov~\cite{Gr81}.

A well-studied example of a geometric property is having polynomial growth.  Specifically,  a finitely generated group $\Gamma$ has \emph{polynomial growth of degree $d$} if there is a finite generating set $\mathcal A \subset \Gamma$ such that the Cayley graph $\mathcal C(\Gamma, \mathcal A)$ has polynomial growth of degree $d$ in the sense of Definition~\ref{def:growth}. As examples, the infinite cyclic group $\Z$ has linear growth, $\Z\oplus\Z$ has quadratic growth since the square grid is its Cayley graph with respect to the standard generating set, and the free group $F_k$ of rank $k\geq 2$ has exponential growth as its Cayley graph with respect to a free generating set is the infinite tree where each vertex has degree $2k$.   It is an outstanding deep  result of Gromov  that for finitely generated groups, having polynomial growth is equivalent to contain a nilpotent finite index subgroup. For a brief overview of results in the area, we refer the reader to~\cite[Page 148]{BrHa99}.

By Theorem~\ref{thm:qi}, in the class of bounded degree graphs, the property of having  $n^d$-containment property is preserved by quasi-isometry. This defines a geometric property of groups.

\begin{definition}
Let $\Gamma$ be a finitely generated group.
We say that $\Gamma$ has the \emph{$n^d$-containment property} if for a (and hence any) finitely generating set $\mathcal A$, the  Cayley graph $\mathcal C(\Gamma, \mathcal A)$ has the $n^d$-containment property.
In this case, if $d=0$, we say that $\Gamma$ has the constant containment property.
\end{definition}

The fact that every finitely generated group is a quotient of a free group implies that their growth is at most exponential. By Theorem~\ref{fImpliesContains},  every finitely generated group satisfies the exponential containment property. Finitely generated groups can have better containment properties. For example, the group $\Z\oplus \Z$ has quadratic growth and, by Corollary~\ref{cor:quadratic},  it satisfies the constant containment property.

\subsection{Crystallographic groups and Uniform tilings}\label{sec:tilings}

A \emph{Euclidean wallpaper group},   also known as a \emph{plane  crystallographic group}, is a subgroup $\Gamma$ of the group of isometries of the Euclidean plane  $\mathbb E^2$ that acts properly  and cocompactly on $\mathbb E^2$.  In this case, the  group $\Gamma$ contains two linearly independent translations that generate a finite index subgroup of $\Gamma$ isomorphic to $\Z \oplus \Z$. Up to group isomorphism, there are seventeen wallpaper groups. There are several references in the area, we refer the reader to~\cite[Page 40]{GuSh87} and the references therein.  The \v Svark-Milnor lemma implies that any wallpaper group is quasi-isometric to the Euclidean plane $\mathbb E^2$. In particular, all these groups are quasi-isometric to the square grid and hence they have quadratic growth. Then Corollary~\ref{cor:quadratic} implies that every wallpaper group has constant containment property.

A \emph{uniform tiling of $\mathbb E^2$} is a tessellation of the plane by regular polygons such that for any two vertices there is an isometry of  $\mathbb E^2$ that preserves the tiling and maps one of the vertices to the other.  There are eleven distinct uniform tilings of the Euclidean plane, see for example~\cite[Page 63]{GuSh87}. The subgroup of isometries of the $\mathbb E^2$ that preserves the tiling is a wallpaper group.  Hence   the underlying graph of the tiling is quasi-isometric to a wallpaper group and hence it has the constant containment property.

Analogously, a \emph{crystallographic group on $\mathbb H^2$}  is a subgroup $\Gamma$ of the group of isometries of the hyperbolic plane  whose action is proper  and cocompact.   It is well known that $\Gamma$ contains a subgroup isomorphic to the free group in two generators, for example as a consequence of the Tits alternative~\cite{Ti72}.  Since the free group of rank $2$ does not have polynomial (or even subexponential) containment property, if follows that $\Gamma$  does not have it as well.

Similarly, a \emph{uniform tiling of $\mathbb H^2$} is a tessellation of the hyperbolic plane by regular (hyperbolic) polygonal faces such that for any two vertices there is an automorphism of the tiling mapping one of the vertices to the other.  There are infinitely many  distinct uniform tilings of the hyperbolic plane, see for example~\cite[Page 261]{Co08} and the references therein. The subgroup of isometries of   $\mathbb H^2$ that preserve the tiling is a crystallographic group.  It follows that the underlying graph of the tiling is quasi-isometric to a crystallographic group on $\mathbb H^2$ and hence it does not have polynomial (or even subexponential) containment property.


\begin{thebibliography}{10}
	
	\bibitem{BrHa99}
	Martin~R. Bridson and Andr{\'e} Haefliger.
	\newblock {\em Metric spaces of non-positive curvature}, volume 319 of {\em
		Grundlehren der Mathematischen Wissenschaften [Fundamental Principles of
		Mathematical Sciences]}.
	\newblock Springer-Verlag, Berlin, 1999.
	
	\bibitem{Co08}
	John~H. Conway, Heidi Burgiel, and Chaim Goodman-Strauss.
	\newblock {\em The symmetries of things}.
	\newblock A K Peters, Ltd., Wellesley, MA, 2008.
	
	\bibitem{DeHa07}
	M.~Develin and S.~G. Hartke.
	\newblock Fire containment in grids of dimension three and higher.
	\newblock {\em Discrete Appl. Math.}, 155(17):2257--2268, 2007.
	
	\bibitem{FH13}
	Ohad~N. Feldheim and Rani Hod.
	\newblock 3/2 firefighters are not enough.
	\newblock {\em Discrete Appl. Math.}, 161(1-2):301--306, 2013.
	
	\bibitem{FHS00}
	S.~Finbow, B.~Hartnell, Q.~Li, and K.~Schmeisser.
	\newblock On minimizing the effects of fire or a virus on a network.
	\newblock {\em J. Combin. Math. Combin. Comput.}, 33:311--322, 2000.
	\newblock Papers in honour of Ernest J. Cockayne.
	
	\bibitem{FM09}
	Stephen Finbow and Gary MacGillivray.
	\newblock The firefighter problem: a survey of results, directions and
	questions.
	\newblock {\em Australas. J. Combin.}, 43:57--77, 2009.
	
	\bibitem{Fo03}
	P.~Fogarty.
	\newblock Catching the fire on grids.
	\newblock Master's thesis, University of Vermont, 2003.
	
	\bibitem{GKP14}
	Tom{\'a}{\v{s}} Gaven{\v{c}}iak, Jan Kratochv{\'{\i}}l, and Pawe{\l} Pra{\l}at.
	\newblock Firefighting on square, hexagonal, and triangular grids.
	\newblock {\em Discrete Math.}, 337:142--155, 2014.
	
	\bibitem{Gr81}
	Mikhael Gromov.
	\newblock Groups of polynomial growth and expanding maps.
	\newblock {\em Inst. Hautes \'Etudes Sci. Publ. Math.}, (53):53--73, 1981.
	
	\bibitem{GuSh87}
	Branko Gr{\"u}nbaum and G.~C. Shephard.
	\newblock {\em Tilings and patterns}.
	\newblock W. H. Freeman and Company, New York, 1987.
	
	\bibitem{Me04}
	M.~E. Messinger.
	\newblock Firefighting on infinite grids.
	\newblock Master's thesis, Dalhousie University, 2004.
	
	\bibitem{Me07}
	M.~E. Messinger.
	\newblock Firefighting on the triangular grid.
	\newblock {\em J. Combin. Math. Combin. Comput.}, 63:37--45, 2007.

	\bibitem{MP17}
	Eduardo Mart\'inez-Pedroza.
	\newblock A note on the relation between Hartnell's firefighter problem and growth of groups.   
	\newblock \emph{Actes du {S}\'eminaire de {T}h\'eorie {S}pectrale et {G}\'eometrie.} Univ. Grenoble I, Saint-Martin-d'H\`eres,
	\newblock 2017 (To appear).
	
	\bibitem{NR08}
	K.~L. Ng and P.~Raff.
	\newblock A generalization of the firefighter problem on {$\Bbb Z\times\Bbb
		Z$}.
	\newblock {\em Discrete Appl. Math.}, 156(5):730--745, 2008.
	
	\bibitem{Ti72}
	J.~Tits.
	\newblock Free subgroups in linear groups.
	\newblock {\em J. Algebra}, 20:250--270, 1972.
	
	\bibitem{MW02}
	Ping Wang and Stephanie~A. Moeller.
	\newblock Fire control on graphs.
	\newblock {\em J. Combin. Math. Combin. Comput.}, 41:19--34, 2002.
	
\end{thebibliography}
\end{document}